\documentclass{aims}
\usepackage{amsmath}
  \usepackage{paralist}
\usepackage{graphicx}  
\usepackage{epstopdf}
 \usepackage[colorlinks=true]{hyperref}
\hypersetup{urlcolor=blue, citecolor=red}

\def\currentvolume{X}
 \def\currentissue{X}
  \def\currentyear{200X}
   \def\currentmonth{XX}
    \def\ppages{X--XX}

\newtheorem{theorem}{Theorem}[section]

\newtheorem{lemma}[theorem]{Lemma}

\theoremstyle{definition}
\newtheorem{definition}[theorem]{Definition}
\newtheorem{remark}{Remark}

\newcommand{\R}{\mathbb{R}}

\newcommand{\vecE}{{\boldsymbol{E}}}
\newcommand{\vecf}{{\boldsymbol{f}}}
\newcommand{\vecg}{{\boldsymbol{g}}}

\newcommand{\vecH}{{\boldsymbol{H}}}
\newcommand{\vecu}{{\boldsymbol{u}}}
\newcommand{\vecw}{  {\boldsymbol{w}} }
\newcommand{\vecv}{  {\boldsymbol{v}}  }

\newcommand{\vecp}{  {\boldsymbol{p}}  }
\newcommand{\vecq}{  {\boldsymbol{q}}  }

\newcommand{\vecL}{{\boldsymbol{L}}}

\newcommand{\vecU}{{\boldsymbol{U}}}

\newcommand{\bz}{{\boldsymbol{z}}}
\newcommand{\by}{{\boldsymbol{y}}}
\newcommand{\bx}{{\boldsymbol{x}}}
\newcommand{\be}{{\boldsymbol{e}}}

\newcommand{\mdiv}{\mbox{div\,}}
\newcommand{\mcurl}{\mbox{curl\,}}
\newcommand{\smcurl}{\overrightarrow{\mbox{curl}}_\Sigma\,}
\newcommand{\tG}{\mathbb{G}}

\title[A Sampling Type Method in an Electromagnetic Waveguide] 
      {A Sampling Type Method in an Electromagnetic Waveguide}
\author[Shixu Meng]{}

 \email{shixumeng@amss.ac.cn}

\begin{document}
\maketitle

\centerline{\scshape Shixu Meng}
\medskip
{\footnotesize
 \centerline{Institute of Applied Mathematics}
   \centerline{Academy of Mathematics and Systems Science}
      \centerline{Chinese Academy of Sciences}
   \centerline{ Beijing, 100190, China}
}

\bigskip


\begin{abstract}
We propose a sampling type method to image scatterer in an electromagnetic waveguide. The waveguide terminates at one end and the measurements are on the other end and in the far field.  The imaging function is based on  integrating the measurements and a known function over the measurement surface directly. The design and analysis of such imaging function are based on a factorization of a data operator given by the measurements. We show by analysis that the imaging function peaks inside the scatterer, where the coercivity of the factorized operator and the design of the known function play a central role. Finally, numerical examples are provided to demonstrate the performance of the imaging method.
\end{abstract}

\section{Introduction}  \label{Introduction}
Inverse scattering  plays an important role in non-destructive testing, medical imaging, geophysical exploration and numerous problems associated with target identification. There have been recent interests in inverse scattering for waveguides, mainly motivated by their numerous applications in ocean acoustics,  non-destructive testing of slender structures, imaging in and of tunnels \cite{baggeroer1993overview,rizzo2010ultrasonic,haack1995state}.
\cite{xu2000generalized} studied the generalized dual space indicator method for underwater imaging. The linear sampling method and other related sampling methods were studied in acoustic waveguides \cite{bourgeois2008linear,bourgeois2012use} and in elastic waveguides \cite{bourgeois2011use,bourgeois2013use}. We also mention the sampling methods for acoustic waveguides studied in \cite{monk2012sampling,arens2011direct,borcea2019factorization}. More recently, \cite{bourgeois2014identification} developed the sampling methods for identification of defects in a periodic waveguide. \cite{monk2016inverse} developed the linear sampling method for an acoustic waveguide in the time domain. \cite{monk2019near} investigated the linear sampling method in an electromagnetic waveguide. We also mention the inverse scattering in an acoustic waveguide \cite{sun2013reconstruction}, the time migration imaging method in an acoustic terminating waveguide \cite{tsogka2017imaging},  in an electromagnetic waveguide \cite{chen2017direct} and the time reversal imaging in an electromagnetic terminating waveguide \cite{borcea2015imaging}.

Our goal is to design a robust sampling type method for an electromagnetic terminating waveguide. Two sampling methods, the linear sampling method and factorization method, have been developed in both free space and waveguides, see \cite{colton2012inverse,cakoni2016qualitative,kirsch2008factorization,cakoni2016inverse} for more details. Such sampling method uses the measurements to define a data operator (in the far field or the near field), and aims to image the scatterer by solving a linear integral equation at each sampling point over a searching domain. The solutions have different generic properties for sampling points inside and outside the scatterer, and therefore allow us to design an imaging function to determine the scatterer. To solve the linear integral equation, which is an ill-posed problem, one needs to apply certain regularization techniques \cite{colton2003linear,cakoni2011linear}. For the full Maxwell's equations, the regularization may burn some computations. Such methods also require a priori estimate on the measurement noise. An alternative way, in contrast to solving the linear integral equation, is to use the data operator directly. Here we mention the orthogonality sampling \cite{potthast2010study,griesmaier2011multi} and a novel sampling method \cite{liu2017novel} developed in the free space. 

In the electromagnetic waveguides, there have been few work \cite{chen2017direct,borcea2015imaging,monk2019near}. In particular, \cite{chen2017direct}  considered a reverse time migration imaging method, where the Helmholtz-Kirchhoff identity was used to show that the imaging function peaks in the scatterer. \cite{borcea2015imaging} applied a time reversal imaging method with optimization. \cite{monk2019near} considered a linear sampling method and also discussed the generalized linear sampling method. Our work follows closely \cite{monk2019near}. Instead of solving a linear integral equation, we propose to design a sampling type method using a  data operator $\mathcal{N}$ and its factorization $\overline{\mathcal{H}^* \overline{\mathcal{T}\mathcal{H} }}$. The factorization provides us insights and powerful analytic tools to show that our imaging function behaves like $\|\mathcal{H} \Psi(\cdot;\bz) \be_j\|$ (with $L^2$-norm in the scatterer) for some ``test function" $\Psi(\cdot;\bz)\be_j$. The analysis of the imaging function is then demonstrate by the behavior of $\mathcal{H} \Psi(\cdot;\bz) \be_j$, which will be proved to be similar to the first derivative of the dyadic Green function.

We now introduce the inverse scattering problem in a heuristic setting. For any $\bx \in \R^3$, we use the following Cartesian coordinate representation $\bx = (x_1,x_2,x_3)$. The waveguide is denoted by $W:=\Sigma \times (-\infty, 0)$, where it has a { rectangular} cross-section $\Sigma:=(0,a)\times (0,b)$, extends to $-\infty$, and has a terminating end at $x_3=0$; here $a$ and $b$ are both positive. We assume that the waveguide is perfectly conducting. The waveguide is filled with some background isotropic homogeneous medium (such as air or vacuum) with electric permittivity $\epsilon_0$ and magnetic permeability $\mu_0$, where $\epsilon_0$ and $\mu_0$ are positive constants. We consider time-harmonic scattering at fixed frequency $\omega$. 

An electric point source $\vecE^i$ at $\by$ with polarization $\vecp$ satisfies the following Maxwell's equations
\begin{eqnarray}
-i k \epsilon_0 \vecE^i - \mcurl \vecH^i = \frac{1}{ik\mu_0} \vecp\,\delta(\by) \quad &\mbox{in}&\quad W, \label{forward source em 1}\\
i k \mu_0 \vecH^i + \mcurl \vecE^i = \boldsymbol{0} \quad &\mbox{in}&\quad W, \label{forward source em 2}\\
\nu \times \vecE^i = \boldsymbol{0} \quad &\mbox{on}&\quad \partial W,\label{forward source em 3}
\end{eqnarray}
where $\vecH^i$ is the corresponding magnetic field, $k:=\omega \sqrt{\epsilon_0\mu_0}$ denotes the background wavenumber, and $\nu$ denotes the unit outward normal to $\partial W$. The electric point source further satisfies a suitable radiation condition when $x_3 \to -\infty$. Such radiation condition will be discussed in details in Section \ref{pre 1}.

The scatterer is denoted by $D$. The scatterer is filled with an isotropic  material with electric permittivity $\widetilde{\epsilon}(\bx)$, constant magnetic permeability $\mu_0$, and electric conductivity $\widetilde{\sigma}(\bx)$. The relative electric permittivity $\epsilon(\bx)$ is given by  $\epsilon(\bx):=\frac{\widetilde{\epsilon}(\bx)}{\epsilon_0} + i \frac{\widetilde{\sigma}(\bx)}{\epsilon_0}$ in $D$ and $\epsilon(\bx):=1$ in $W \backslash \overline{D}$ respectively.

The total electric wave field $\vecE$ and total magnetic wave field $\vecH$ satisfy
 \begin{eqnarray}
-i k \epsilon \vecE - \mcurl \vecH = \frac{1}{ik\mu_0} \vecp\,\delta(\by) \quad &\mbox{in}&\quad W, \label{forward total em 1}\\
i k \mu_0 \vecH + \mcurl \vecE = \boldsymbol{0} \quad &\mbox{in}&\quad W, \label{forward total em 2}\\
\nu \times \vecE= \boldsymbol{0} \quad &\mbox{on}&\quad \partial W,\label{forward total em 3}
\end{eqnarray}
where the electric wave field again satisfies a suitable radiation condition when $x_3 \to -\infty$.

It is convenient to write down the Maxwell's equations in terms of the electric wave fields only. We can rewrite equations \eqref{forward source em 1} -- \eqref{forward source em 3} in terms of the electric point source as
\begin{eqnarray}
\mcurl^2 \vecE^i - k^2 \vecE^i =  \vecp\,\delta(\by) \quad &\mbox{in}&\quad W, \label{forward source e 1}\\
\nu \times \vecE^i = \boldsymbol{0} \quad &\mbox{on}&\quad \partial W.\label{forward source e 2}
\end{eqnarray}
Similarly, equations \eqref{forward total em 1} -- \eqref{forward total em 3} in terms of the electric scattered wave field $\vecE^s: = \vecE - \vecE^i$ read
\begin{eqnarray}
\mcurl^2 \vecE^s - k^2 \epsilon \vecE^s =  k^2 (\epsilon -1) \vecE^i \quad &\mbox{in}&\quad W, \label{forward scattered e 1}\\
\nu \times \vecE^s = \boldsymbol{0} \quad &\mbox{on}&\quad \partial W.\label{forward scattered e 2}
\end{eqnarray}
Equations \eqref{forward source e 1}--\eqref{forward source e 2} and \eqref{forward scattered e 1}--\eqref{forward scattered e 2} are further complimented by the radiation condition when $x_3 \to -\infty$.

We now introduce the inverse problem. Denote by $\vecE^i(\cdot;\by;\vecp)$ the electric point source at $\by$ with polarization $\vecp$ and $\vecE^s(\cdot;\by;\vecp)$ the corresponding electric scattered wave field satisfying \eqref{forward scattered e 1}--\eqref{forward scattered e 2} with $\vecE^i=\vecE^i(\cdot;\by;\vecp)$. Denote by $\Sigma_r:=\Sigma \times \{r\}$, the measurement surface away from the scatterer. The \textbf{inverse problem} is to determine the scatterer $D$ from $\{ \vecE^s(\bx;\by;\vecp)\}$ for all $\bx \in \Sigma_r$, $\by \in \Sigma_r$,  and all polarization $\vecp$.

The remaining of the paper is organized as follows. We first discuss the radiation condition in an electromagnetic waveguide and the corresponding dyadic Green function in Section \ref{pre 1}. We then discuss in Section \ref{pre 2} the forward scattering problem, and introduce the data operator $\mathcal{N}$ and its factorization $\overline{\mathcal{H}^* \overline{\mathcal{T}\mathcal{H} }}$. We use such factorization to design and analyze a sampling type method in Section \ref{Imaging function}. In particular, we show that the imaging function behaves like $\|\mathcal{H} \Psi(\cdot;\bz) \be_j\|_{\vecL^2(D)}$   for some ``test function" $\Psi(\cdot;\bz)\be_j$. We then show that $\|\mathcal{H} \Psi(\cdot;\bz) \be_j\|_{\vecL^2(D)}$ peaks for sampling points $\bz$ inside the scatterer, and so does the imaging function. We further give a modal representation of the imaging function. Numerical examples are provided in Section \ref{Numerics} to demonstrate the performance of the imaging method. We finally conclude our paper in Section \ref{conclusion}.
\section{Radiation Condition and Dyadic Green Function} \label{pre 1}
\subsection{Radiation Condition}
Let us introduce the propagating and evanescent modes in the { rectangular waveguide $W$. Recall that the cross-section is $\Sigma = (0,a)\times(0,b)$.} We denote by $\widehat{\bx}:=(x_1,x_2)$ the first two component of $\bx=(x_1,x_2,x_3)$, and $\partial_j$ the partial derivative respect to $x_j$.

{
Let $\{(u_m,\lambda_m)\}_{m=0}^\infty$ and  $\{(v_n,\mu_n)\}_{n=1}^\infty$ be defined by
\begin{equation*}
\Bigg\{
\begin{array}{c}
 u_m(\widehat{x}) := \cos(\frac{m_1 \pi x_1}{a}) \cos(\frac{m_2 \pi x_2}{b})  \\
 \\
\lambda_m^2 := (\frac{m_1 \pi}{a})^2 + (\frac{m_2 \pi}{b})^2
\end{array}, \quad \mbox{and} \quad 
\Bigg\{
\begin{array}{c}
 v_n(\widehat{x}) := \sin(\frac{n_1 \pi x_1}{a}) \sin(\frac{n_2 \pi x_2}{b})  \\
 \\
\mu_n^2 := (\frac{n_1 \pi}{a})^2 + (\frac{n_2 \pi}{b})^2
\end{array},
\end{equation*}
where we have associated each $(m_1,m_2)$ with a unique $m$, and each $(n_1,n_2)$ with a unique $n$, such that $0\le\lambda_{m} \le \lambda_{m+1}$ and $0<\mu_{n} \le \mu_{n+1}$.
}
We now define
\begin{equation} \label{pre modes}
\left\{
\begin{array}{c}
M_m(\bx) := \Bigg(
\begin{array}{c}
  \partial_2 u_m(\widehat{x}) \\
 -\partial_1 u_m(\widehat{x})   \\
    0
\end{array}
\Bigg ) e^{i h_m x_3}\\
P_n(\bx) := \frac{1}{k}\Bigg(
\begin{array}{c}
 i g_n \partial_1 v_n(\widehat{x}) \\
 i g_n \partial_2 v_n(\widehat{x})   \\
    0
\end{array}
\Bigg ) e^{i g_n x_3}
\\Q_n(\bx) := \frac{1}{k} \Bigg(
\begin{array}{c}
0 \\
0  \\
    \mu_n^2 v_n
\end{array}
\Bigg ) e^{i g_n x_3}
\end{array}
\right.,
\end{equation}
where  $h_m$ is defined by $h_m := \sqrt{k^2-\lambda_m^2}$ with branch cut in $\{z: \Im z \ge 0\}$ and $g_n$ is defined by
$g_n := \sqrt{k^2-\mu_n^2}$ with branch cut in $\{z: \Im z \ge 0\}$. 

When $g_m$ and $h_n$ are real-valued, $M_m(\bx^-)$, $P_n(\bx^-)$ and $Q_n(\bx^-)$ are propagating modes that propagates along the waveguide axis to $-\infty$, where $\bx^-:=(x_1,x_2,-x_3)$ for any $\bx \in \R^3$; for a fixed wavenumber, there are at most finitely many propagating modes. We remark that when $m=0$, $M_m$ vanishes. When $g_m$ and $h_n$ are imaginary-valued, $M_m(\bx^-)$, $P_n(\bx^-)$ and $Q_n(\bx^-)$ are evanescent modes that decay as $x_3 \to -\infty$. We assume that $k^2\not=\lambda_m^2$ and $k^2 \not=\mu_n^2$ for any $m$ and $n$.

Any solution to the Maxwell's equation for $x_3 \ll -1$ can be represented by the superposition of the propagating modes and evanescent modes. A direct calculation yields that the modes $P_n(\bx^-)$ and $-Q_n(\bx^-)$ must have the same coefficients. For a more detailed discussion, we refer to \cite{monk2019near,borcea2015imaging} and the reference therein.  We now introduce the following radiation condition.
\begin{definition} \label{ORC}
The electric wave field $\vecE$ is said to satisfy the outgoing radiation condition if, for $x_3 \ll -1$, $\vecE$ is superposition of  propagating modes and evanescent modes,
\begin{equation*}
\vecE(\bx) = \sum_{m=1}^\infty a_m M_m(\bx^-) + \sum_{n=1}^\infty b_n \big[P_n(\bx^-)-Q_n(\bx^-)\big],
\end{equation*}
with constants $a_m$ and $b_n$ determined by $\vecE$.
\end{definition}
\subsection{Dyadic Green Function}
We first introduce the electric dyadic Green function $\widetilde{\tG}_e(\bx;\by)$ for the full waveguide $\widetilde{W} = \Sigma \times (-\infty, \infty)$. In the full waveguide $\widetilde{W}$, $\widetilde{\tG}_e(\bx;\by)$ satisfies 
\begin{eqnarray*}
\mcurl_{\bx}^2 \widetilde{\tG}_e - k^2 \widetilde{\tG}_e =  \delta(\by) \boldsymbol{I} \quad &\mbox{in}&\quad \widetilde{W}, \label{dyadic full 1}\\
\nu \times \widetilde{\tG}_e = \boldsymbol{0} \quad &\mbox{on}&\quad \partial \widetilde{W}.\label{dyadic full 2}
\end{eqnarray*}
From \cite{chen2017direct}, we can directly obtain that
\begin{equation} \label{full G modal}
\widetilde{\tG}_e(\bx;\by) = 
\left\{
\begin{array}{cc}
\hspace{-2.9cm}\sum_{m=1}^\infty c_m M_m(\bx) M_m^T(\by^-) &  \\ + \sum_{n=1}^\infty d_n [ P_n(\bx) + Q_n(\bx)] [ P_n(\by^-) - Q_n(\by^-)]{^T},   &x_3 >y_3, \\
& \\
 \hspace{-2.9cm}\sum_{m=1}^\infty c_m M_m(\bx^-) M_m^T(\by) &\\+ \sum_{n=1}^\infty d_n [ P_n(\bx-) - Q_n(\bx^-)] [ P_n(\by) + Q_n(\by)]{^T},  &      x_3 <y_3,
\end{array}
\right.
\end{equation}
where $T$ denotes the standard transpose (no conjugate), and 
\begin{equation*}
c_m:= i\frac{1}{2h_m \lambda_m^2}, \qquad d_n:= -i\frac{1}{2g_n \mu_n^2}.
\end{equation*}
Now  the electric dyadic Green function in $W$ satisfies
\begin{eqnarray}
\mcurl_{\bx}^2 \tG_e - k^2 \tG_e =  \delta(\by) \boldsymbol{I} \quad &\mbox{in}&\quad W, \label{dyadic 1}\\
\nu \times \tG_e = \boldsymbol{0} \quad &\mbox{on}&\quad \partial W.\label{dyadic 2}
\end{eqnarray}
From the dyadic Green function given by \eqref{full G modal} in the full waveguide, we can directly obtain the following modal representation of $\tG_e$
\begin{equation} \label{G modal}
\tG_e(\bx;\by) = 
\left \{
\begin{array}{cc}

\begin{array}{cc}
\hspace{-2.0cm}\sum_{m=1}^\infty c_m  [M_m(\bx) -M_m(\bx^-)] M_m^T(\by^-) &\\ \hspace{-0.cm}+ \sum_{n=1}^\infty d_n \Big( [ P_n(\bx) - P_n(\bx^-)] &\\
+[ Q_n(\bx) + Q_n(\bx^-)] \Big) [ P_n(\by^-) - Q_n(\by^-)]^T,    ~~x_3 >y_3, &
\end{array} \\
& \\
\begin{array}{cc}
\hspace{-2.0cm}\sum_{m=1}^\infty c_m  M_m(\bx^-) [M_m^T(\by) -M_m^T(\by^-)]  &\\ \hspace{-0.5cm}+\sum_{n=1}^\infty d_n [ P_n(\bx^-) - Q_n(\bx^-)]    \Big( [ P_n(\by) - P_n(\by^-)]^T &\\
+ [ Q_n(\by) + Q_n(\by^-)]^T \Big),   ~~x_3 <y_3.&
\end{array}
\end{array}
\right.
\end{equation}
The above electric dyadic Green function satisfies the outgoing radiation condition in Definition \ref{ORC}. It is directly verified that the electric dyadic Green function satisfies the reciprocity relation
\begin{equation*}
\tG_e^T(\bx;\by) = \tG_e(\by;\bx).
\end{equation*}
We have immediately that the electric point source in \eqref{forward source e 1} -- \eqref{forward source e 2} is given by $\vecE^i (\cdot;\by;\vecp)= \tG_e(\cdot;\by) \vecp$.

\section{Forward Problem and Factorization of Operators} \label{pre 2}
\subsection{Forward Problem}
To begin with, we make more precise about the waveguide $W$ and scatterer $D$. Assume that the relative electric permittivity $\epsilon$  belongs to $L^\infty (W)$ and  is bounded below by some positive constant. Assume that $D$ is a bounded, open, Lipschitz domain.

We now introduce the following standard Sobolev spaces. For any bounded Lipschitz domain $\Omega \in \R^3$, we denote by $\vecL^2(\Omega) := (L^2(\Omega))^3$ and
\begin{eqnarray*}
\vecH(\mcurl, \Omega) := \{ \vecu \in \vecL^2(\Omega): \mcurl \vecu \in \vecL^2(\Omega) \}.
\end{eqnarray*}
Furthermore $\vecH_{loc}(\mcurl, W)$ denotes the corresponding local space for the unbounded waveguide $W$.

We look for an outgoing radiating solution $\vecE^s \in \vecH_{loc}(\mcurl, W)$ to the forward problem \eqref{forward scattered e 1} -- \eqref{forward scattered e 2}.
The well-posedness of the forward problem \eqref{forward scattered e 1} -- \eqref{forward scattered e 2} with $W$ replaced by the full waveguide $\widetilde{W}$ has been proved in \cite{monk2019near}. They proved that if $\Im \epsilon$ is bounded below in some open bounded subdomain of $D$ with non-zero measure, then the forward problem  is well-posed for any real valued wavenumber $k$; if $\Im \epsilon$ vanishes in $D$, then the forward problem is well-posed except for, at most, a discrete set of real $k$ values whose only possible accumulation point is $\infty$. Our half-waveguide scattering problem \eqref{forward scattered e 1} -- \eqref{forward scattered e 2} can be studied exactly in the same way. See also \cite{borcea2015imaging} on the well-posedness of the forward problem in a half-waveguide. We summarize the above argument as a lemma.

\begin{lemma}
The forward problem is well-posed except for, at most, a discrete set of real $k$ values whose only possible accumulation point is $\infty$.
\end{lemma}
We will always choose wavenumber $k$ such that the forward problem \eqref{forward scattered e 1} -- \eqref{forward scattered e 2} is well-posed throughout our context.

For later purposes, let $W_s = \Sigma \times (s,0)$ and $\Sigma_s=\Sigma \times \{s\}$, and we denote the vector fields that are tangential to $\Sigma_s$ by
\begin{equation*}
\vecL^2_t(\Sigma_s):=\{ \vecu \in \vecL^2(\Sigma_s): \vecu \cdot \nu=0 \},
\end{equation*}
where $\nu$ denotes the unit outward normal to $\Sigma_s$. We then have the standard $\vecH^{1/2}_t(\Sigma_s)$, which includes all vector fields in $\vecH^{1/2}(\Sigma_s)$ that are tangential to $\Sigma_s$. Let $\widetilde{\vecH}^{-1/2}_t(\Sigma_s)$ be the dual of  $\vecH^{1/2}_t(\Sigma_s)$. We define
\begin{eqnarray*}
\widetilde{\vecH}^{-1/2}(\mdiv, \Sigma_s):=\{ \vecf \in \widetilde{\vecH}^{-1/2}_t(\Sigma_s): \vecf=\sum_{m=1}^\infty a_m \nabla_{\Sigma} u_m + \sum_{n=1}^\infty b_n \smcurl v_n,\\
\sum_{m=1}^\infty |a_m|^2 |\lambda_m|^3 + \sum_{n=1}^\infty |b_n|^2 |\mu_n| < \infty \}.
\end{eqnarray*}
The dual space of $\widetilde{\vecH}^{-1/2}(\mdiv, \Sigma_s)$ is denoted by $\widetilde{\vecH}^{-1/2}(\mcurl, \Sigma_s)$. Therefore for any outgoing radiating solution $\vecw \in \vecH(\mcurl,W_s)$ to the forward problem \eqref{forward scattered e 1} -- \eqref{forward scattered e 2}, $ \vecw \times \nu|_{\Sigma_s} \in \widetilde{\vecH}^{-1/2}(\mdiv, \Sigma_s)$  and $ \big(\nu\times \mcurl\vecw \big) \times \nu|_{\Sigma_s} \in \widetilde{\vecH}^{-1/2}(\mcurl, \Sigma_s)$. We refer to \cite{monk2019near} for more details.
\begin{remark} \label{traces modal}
For any outgoing radiating solution $\vecw$ satisfying Definition \ref{ORC}, if its trace on $\Sigma_s$ (where $\Sigma_s$ is such that $D \subset W_s$) is given by
\begin{equation*}
\vecw \times \nu|_{\Sigma_s} = \sum_{m=1}^\infty a_m \nabla_\Sigma u_m(\widehat{\bx}) - \frac{1}{k}\sum_{n=1}^\infty i g_n b_n \smcurl v_n(\widehat{\bx})
\end{equation*}
for some constants $a_m$ and $b_n$, then we have directly that
\begin{equation*}
(\nu \times \mcurl \vecw) \times \nu|_{\Sigma_s} = \sum_{m=1}^\infty a_m (-ih_m) \nabla_\Sigma u_m(\widehat{\bx}) +  \frac{1}{k} \sum_{n=1}^\infty i g_n b_n (ig_n + \frac{-\mu_n^2}{ig_n} ) \smcurl v_n(\widehat{\bx}).
\end{equation*}
\end{remark}

\subsection{Factorization of Data Operator} \label{fac method}
Recall that $\vecE^i(\cdot;\by;\vecp)$ is electric point source at $\by$ with polarization $\vecp$ and $\vecE^s(\cdot;\by;\vecp)$ is the corresponding electric scattered wave field satisfying \eqref{forward scattered e 1}--\eqref{forward scattered e 2} with $\vecE^i=\vecE^i(\cdot;\by;\vecp)$.

We first introduce the data operator $\mathcal{N}: \vecL^2_t(\Sigma_r) \to \vecL^2_t(\Sigma_r)$ defined by
\begin{equation} \label{definition operator N}
\big( \mathcal{N} \vecg \big)(\bx):= \Big(\nu(\bx) \times  \int_{\Sigma_r} \vecE^s(\bx;\by;\vecg(\by)) ~d S_\by \Big) \Big|_{\Sigma_r}.
\end{equation}

In the following, we give a factorization of the data operator $\mathcal{N}$. To begin with, we define the bounded linear operator $\mathcal{H}: \vecL^2_t(\Sigma_r) \to \vecH_{inc} (D)$ by
\begin{equation} \label{definition operator H}
\mathcal{H} \vecg := \vecw^i|_D,
\end{equation}
where 
$\vecw^i(\bx):=\int_{\Sigma_r} \tG_e(\bx;\by) \vecg(\by) d S_\by$, $\bx \in W \backslash \Sigma_r$; and $\vecH_{inc} (D)$ is defined by
\begin{equation*}
\vecH_{inc} (D):=\{ \vecu \in \vecL^2(D): \mcurl^2 \vecu -k^2 \vecu=0 \mbox{ in } D \}.
\end{equation*}
{ We note that $\vecH_{inc} (D)$ is a closed subspace of $\vecL^2(D)$ and hence a Hilbert space. Therefore the adjoint of $\mathcal{H}$, given by $\mathcal{H}^*$, can be defined as an operator $\vecH_{inc} (D) \to \vecL^2_t(\Sigma_r)$; more precisely, $\mathcal{H}^*:  \vecH_{inc} (D) \to \vecL^2_t(\Sigma_r)$ is given by}
\begin{equation} \label{definition operator H*}
\big(\mathcal{H}^* \vecv \big) (\bx) :=  \Big(\nu(\bx) \times  \int_D \overline{\tG_e(\bx;\by)} \vecv(\by) ~d\by \Big) \times \nu(\bx) \Big|_{\Sigma_r},
\end{equation}
as we can directly verify this via
\begin{eqnarray*}
\langle \vecv, \mathcal{H} \vecg \rangle &=& \int_D \int_{\Sigma_r} \vecv(\bx)\cdot \big(  \overline{\tG_e(\bx;\by) \vecg(\by)} \big) d S_\by~ d \bx \\
&=&  \int_D \int_{\Sigma_r} \overline{\vecg(\by)}\cdot \big(  \overline{\tG_e(\bx;\by)}^T \vecv(\bx) \big) d S_\by ~d \bx \\
&=&  \int_{\Sigma_r} \int_D \overline{\vecg(\by)}\cdot \big(  \overline{\tG_e(\by;\bx)} \vecv(\bx) \big) d \bx ~d S_\by = \langle \mathcal{H}^* \vecv,\vecg \rangle,
\end{eqnarray*}
where we have applied the reciprocity relation of $\tG_e(\bx;\by)$.

Define the bounded linear operator $\mathcal{T}: \vecH_{inc}(D) \to \vecL^2(D)$ by
\begin{equation} \label{definition operator T}
\mathcal{T} \vecw^i:= k^2(\epsilon-1) (\vecw^i+\vecw^s),
\end{equation}
where $\vecw^s$ is the unique solution to \eqref{forward scattered e 1} -- \eqref{forward scattered e 2} with $\vecE^i$ replaced by $\vecw^i$. We are now ready to state the following theorem. 
\begin{theorem}\label{data operator factorization}
The data operator $\mathcal{N}$ defined via \eqref{definition operator N} can be factorized by
\begin{equation} 
\mathcal{N} \vecg \times \nu = \overline{\mathcal{H}^* \overline{\mathcal{TH} \vecg}},
\end{equation}
for any $\vecg \in \vecL^2_t(\Sigma_r)$. Here $\mathcal{H}$, $\mathcal{H}^*$, and $\mathcal{T}$ are defined via \eqref{definition operator H},  \eqref{definition operator H*}, and  \eqref{definition operator T} respectively.
\end{theorem}
{
\begin{proof}
For any $\vecg \in \vecL^2_t(\Sigma_r)$, we get from the definition of $\mathcal{H}$ in \eqref{definition operator H} that $\mathcal{H}\vecg = \vecw^i$, where $\vecw^i(\bx):=\int_{\Sigma_r} \tG_e(\bx;\by) \vecg(\by) d S_\by$, $\bx \in W \backslash \Sigma_r$. Denote by $\vecw^s$ the unique solution to \eqref{forward scattered e 1} -- \eqref{forward scattered e 2} with $\vecE^i$ replaced by $\vecw^i$. Note that $\vecE^i(\bx;\by;\vecg(\by))=\tG_e(\bx;\by) \vecg(\by)$, then $\vecw^i$ can be written as $\vecw^i(\bx) = \int_{\Sigma_r} \vecE^i(\bx;\by;\vecg(\by))  d S_\by$,  therefore it follows from superposition that
\begin{eqnarray} \label{factorization proof w^s}
\vecw^s(\bx) = \int_{\Sigma_r} \vecE^s(\bx;\by;\vecg(\by))  d S_\by
\end{eqnarray}
since $\vecE^s(\cdot;\by;\vecp)$ is the corresponding electric scattered wave field satisfying \eqref{forward scattered e 1}--\eqref{forward scattered e 2} with $\vecE^i=\vecE^i(\cdot;\by;\vecp)$. 

From the definition of $\mathcal{N}$ in \eqref{definition operator N} and the expression of $\vecw^s$ in \eqref{factorization proof w^s}, to prove our theorem, it is sufficient to show that 
\begin{equation*}
(\nu\times\vecw^s) \times \nu|_{\Sigma_r} = \overline{\mathcal{H}^* \overline{\mathcal{T} \vecw^i}}.
\end{equation*}
From the definition of $\mathcal{T}$ in \eqref{definition operator T}, it is therefore sufficient to show that 
\begin{equation*}
(\nu\times\vecw^s) \times \nu|_{\Sigma_r} = \overline{\mathcal{H}^* \overline{k^2(\epsilon-1) (\vecw^i+\vecw^s)}}.
\end{equation*}
From the definition of $\mathcal{H}^*$ in \eqref{definition operator H*}, it is then sufficient to show that 
\begin{eqnarray}
\vecw^s&=& k^2 \overline{ \int_D \overline{\tG_e(\cdot;\by)} \overline{(\epsilon(\by)-1) (\vecw^i(\by)+\vecw^s(\by))} } ~d\by \\
&=& k^2 \int_D \tG_e(\cdot;\by) (\epsilon(\by)-1) (\vecw^i(\by)+\vecw^s(\by)) ~d\by. \label{factorization proof LS eqn}
\end{eqnarray}
This follows from the fact that $\vecw^s$ is the unique solution to \eqref{forward scattered e 1} -- \eqref{forward scattered e 2} with $\vecE^i$ replaced by $\vecw^i$. Indeed
\begin{eqnarray*}
\mcurl^2 \vecw^s - k^2 \epsilon \vecw^s =  k^2 (\epsilon -1) \vecw^i \quad &\mbox{in}&\quad W
\end{eqnarray*}
is equivalent to 
\begin{eqnarray*}
\mcurl^2 \vecw^s - k^2  \vecw^s =  k^2 (\epsilon -1) (\vecw^i+\vecw^s) \quad &\mbox{in}&\quad W,
\end{eqnarray*}
and therefore $\vecw^s$ can be represented using Lippmann-Schwinger equation as \eqref{factorization proof LS eqn}. This completes the proof.
\end{proof}
}

Alternatively, we can factorize the operator $\mathcal{N}$ in the standard way $\mathcal{N} = \mathcal{S} \mathcal{H}^*\mathcal{T}\mathcal{H}$ with the help of a certain ingoing to outgoing operator $\mathcal{S}$, see for instance the acoustic case in \cite{borcea2019factorization}. The data operator, together with its factorization, { may be explored} in the linear sampling method to image the scatterer by solving a linear integral equation. {For instance, in the full electromagnetic waveguide case \cite{monk2019near}}, the linear integral equation takes the form of $\mathcal{N} \vecg_\bz = \tG_e(\cdot,\bz) \vecq$, where $\tG_e(\cdot,\bz)$ is the electric dyadic Green function at the sampling point $\bz$ and $\vecq$ is some polarization; the solutions $\vecg_\bz$ display different generic properties for sampling points inside and outside the scatterer, and therefore may allow us to design an imaging method to determine the scatterer. To solve the linear integral equation, an ill-posed problem, one needs to apply certain regularization techniques \cite{monk2019near,colton2003linear,cakoni2011linear}. For the full Maxwell's equations, the regularization may burn some computations. 
An alternative way, in contrast to solving a linear integral equation, is to make use of the data operator directly. 
This motivates us to design a robust imaging method  by integrating the measurements and a ``test function" over the measurement surface in the next section.

\section{Imaging Function} \label{Imaging function}
In this section, we make use of the factorization in Section \ref{fac method} to propose a sampling type method. The outline of this section is as follows. We first show that our imaging function behaves like $\sum_{j=1}^3\|\mathcal{H} \Psi(\cdot;\bz) \be_j\|_{\vecL^2(D)}$ for some given ``test function" $\Psi(\cdot;\bz)\be_j$. The analysis of the imaging function is then demonstrated by the behavior of $\mathcal{H} \Psi(\cdot;\bz) \be_j$, proved to be similar to the first derivative of the Green function.

To begin with, let $\langle \cdot, \cdot \rangle$ be the $\vecL^2(\Sigma_r)$ inner product, and let $(\be_1,\be_2,\be_3)$ be the identity matrix. The imaging function is designed by
\begin{equation} \label{imaging function Iz}
I(\bz):=\sum_{j=1}^3 I_j(\bz), \mbox{ with } I_j(\bz):= \big| \langle \overline{\mathcal{N} \Psi (\cdot;\bz) \be_j} \times \nu,  \Psi (\cdot;\bz) \be_j \rangle \big|,
\end{equation}
here the tensor-valued function $\Psi(\by;\bz)$ is given by
\begin{eqnarray} \label{test function Psi}
\Psi(\by;\bz) &:=& \sum_{m=1}^M (-i) \frac{h_{m}}{2 \lambda_{m}^2} \overline{M}_{m} (\by^-) [ \overline{M}_{m}(\bz) - \overline{M}_{m}(\bz^-) ]^T \nonumber \\
&& \hspace{-1cm}+ \sum_{n=1}^N i  \frac{k^2}{2 \mu_{n}^2 g_n}  \overline{P}_{n} (\by^-)   \big( [ \overline{P}_{n}(\bz) - \overline{P}_{n}(\bz^-) ]^T  + [ \overline{Q}_{n}(\bz) + \overline{Q}_{n}(\bz^-) ]^T  \big),
\end{eqnarray}
where $M$ and $N$ are the indices such that $\lambda_{M}<k<\lambda_{M+1}$ and $\mu_{N}<k<\mu_{N+1}$ respectively.
\subsection{Analysis of Imaging Function} \label{Iz to Hpsi z}
Let $\langle \cdot,\cdot \rangle_{\vecL^2(D)}$ be the $\vecL^2(D)$ inner product. We first show the following lemma.
\begin{lemma} \label{lemma I_j reformulation}
$I_j(\bz) = |\langle \mathcal{T H} \Psi(\cdot;\bz) \be_j,   \mathcal{H} \Psi(\cdot;\bz) \be_j\rangle_{\vecL^2(D)}|$ for any $j=1,2,3$.
\end{lemma}
\begin{proof}
From the factorization of $\mathcal{N}$ in Lemma \ref{data operator factorization}, we have that
\begin{eqnarray*}
I_j(\bz) &=& |\langle \mathcal{H}^* \overline{\mathcal{T H} \Psi(\cdot;\bz) \be_j},    \Psi(\cdot;\bz) \be_j \rangle_{\vecL^2(D)}| = |\langle \overline{\mathcal{T H} \Psi(\cdot;\bz) \be_j},   \mathcal{H} \Psi(\cdot;\bz) \be_j\rangle_{\vecL^2(D)}| \\
&=& |\langle \mathcal{T H} \Psi(\cdot;\bz) \be_j,   \overline{\mathcal{H} \Psi(\cdot;\bz) \be_j} \rangle_{\vecL^2(D)}|.
\end{eqnarray*}
To prove the lemma, it is sufficient to show that 
\begin{equation} \label{lemma HPsi real valued}
\overline{\mathcal{H} \Psi(\bx;\bz) \be_j}  = \mathcal{H} \Psi(\bx;\bz) \be_j.
\end{equation}
Indeed, we can obtain from the explicit expression of $\Psi(\by;\bz)$ in \eqref{test function Psi} and the modal representation of $\tG_e(\bx;\by)$ in \eqref{G modal} that
\begin{eqnarray*}
 &&\mathcal{H} \Psi(\bx;\bz) \be_j \\&=& \sum_{m=1}^M\sum_{m'=1}^M c_m(-i) \frac{h_{m'}}{2 \lambda_{m'}^2}[ M_m(\bx) - M_m(\bx^-) ]   [ \overline{M}_{m'}(\bz) - \overline{M}_{m'}(\bz^-) ]^T\be_j  \\
&& \hspace{1.5cm} \cdot \langle M_m(\by^-), M_{m'}(\by^-) \rangle \nonumber \\
&+&  \sum_{n'=1}^N\sum_{n=1}^N d_n i  \frac{k^2}{2 \mu_{n'}^2 g_{n'}}  \Big( [ P_n(\bx) - P_n(\bx^-)] \\
&&+ [ Q_n(\bx) + Q_n(\bx^-)] \Big) \Big( [ \overline{P}_{n'}(\bz) - \overline{P}_{n'}(\bz^-) ]^T + [ \overline{Q}_{n'}(\bz) + \overline{Q}_{n'}(\bz^-) ]^T \Big) \be_j \\
&&\hspace{1.5cm} \cdot \langle \overline{P}_{n} (\by^-) - \overline{Q}_{n} (\by^-), \overline{P}_{n'} (\by^-) \rangle,
\end{eqnarray*}
where we have applied that $M_m$ and $P_n-Q_n$ are orthogonal. From \eqref{pre modes}, we have that
$$
\langle M_m(\by^-), M_{m'}(\by^-) \rangle = \lambda_m^2\delta_{mm'}, \quad \langle \overline{P}_{n} (\by^-) - \overline{Q}_{n} (\by^-), \overline{P}_{n'} (\by^-)  \rangle = \frac{\mu_n^2 g_n^2}{k^2} \delta_{nn'},
$$
and we can then derive that
\begin{eqnarray} \label{HPsi modal}
 \mathcal{H} \Psi(\bx;\bz) \be_j &=& \sum_{m=1}^M  (-i) \frac{c_m h_{m}}{2}[ M_m(\bx) - M_m(\bx^-) ]   [ \overline{M}_{m}(\bz) - \overline{M}_{m}(\bz^-) ]^T\be_j   \nonumber \\
&& \hspace{-2.5cm}+  \sum_{n=1}^N i   \frac{d_n g_{n}}{2}  \Big( [ P_n(\bx) - P_n(\bx^-)] + [ Q_n(\bx) + Q_n(\bx^-)] \Big)  \Big( [ \overline{P}_{n}(\bz) - \overline{P}_{n}(\bz^-) ]^T \nonumber \\
&&+ [ \overline{Q}_{n}(\bz) + \overline{Q}_{n}(\bz^-) ]^T \Big) \be_j.
\end{eqnarray}
From the explicit expressions in \eqref{pre modes}, it is directly verified that the quantity on the right hand side is real-valued. This proves equation \eqref{lemma HPsi real valued} and completes the proof.
\end{proof}

From Lemma \ref{lemma I_j reformulation}, we can directly obtain the following theorem.
\begin{theorem} \label{resolution I_j(z)}
If the operator $\mathcal{T}$ is coercive, then
\begin{equation}
c_1 \|  \mathcal{H} \Psi(\cdot;\bz) \be_j\|^2_{\vecL^2(D)} \le I_j(\bz) \le c_2  \|  \mathcal{H} \Psi(\cdot;\bz) \be_j\|^2_{\vecL^2(D)}
\end{equation}
for some constants $c_1$ and $c_2$ independent of $\bz$.
\end{theorem}
\begin{remark}
We observe from Theorem \ref{resolution I_j(z)} that $I_j(\bz)$ behaves qualitatively as $\|  \mathcal{H} \Psi(\cdot;\bz) \be_j\|^2_{\vecL^2(D)}$. In the following, we will show that  $|\mathcal{H} \Psi (\bx_*;\bz) \be_j|$ peaks when $\bz$ coincide with $\bx_*$. Thus we can conclude that $\|  \mathcal{H} \Psi(\cdot;\bz) \be_j\|^2_{\vecL^2(D)}$ is expected to peak in the scatterer $D$, and so is  $I_j(\bz)$.
\end{remark}
    \begin{figure}[ht!]
\includegraphics[width=0.36\linewidth]{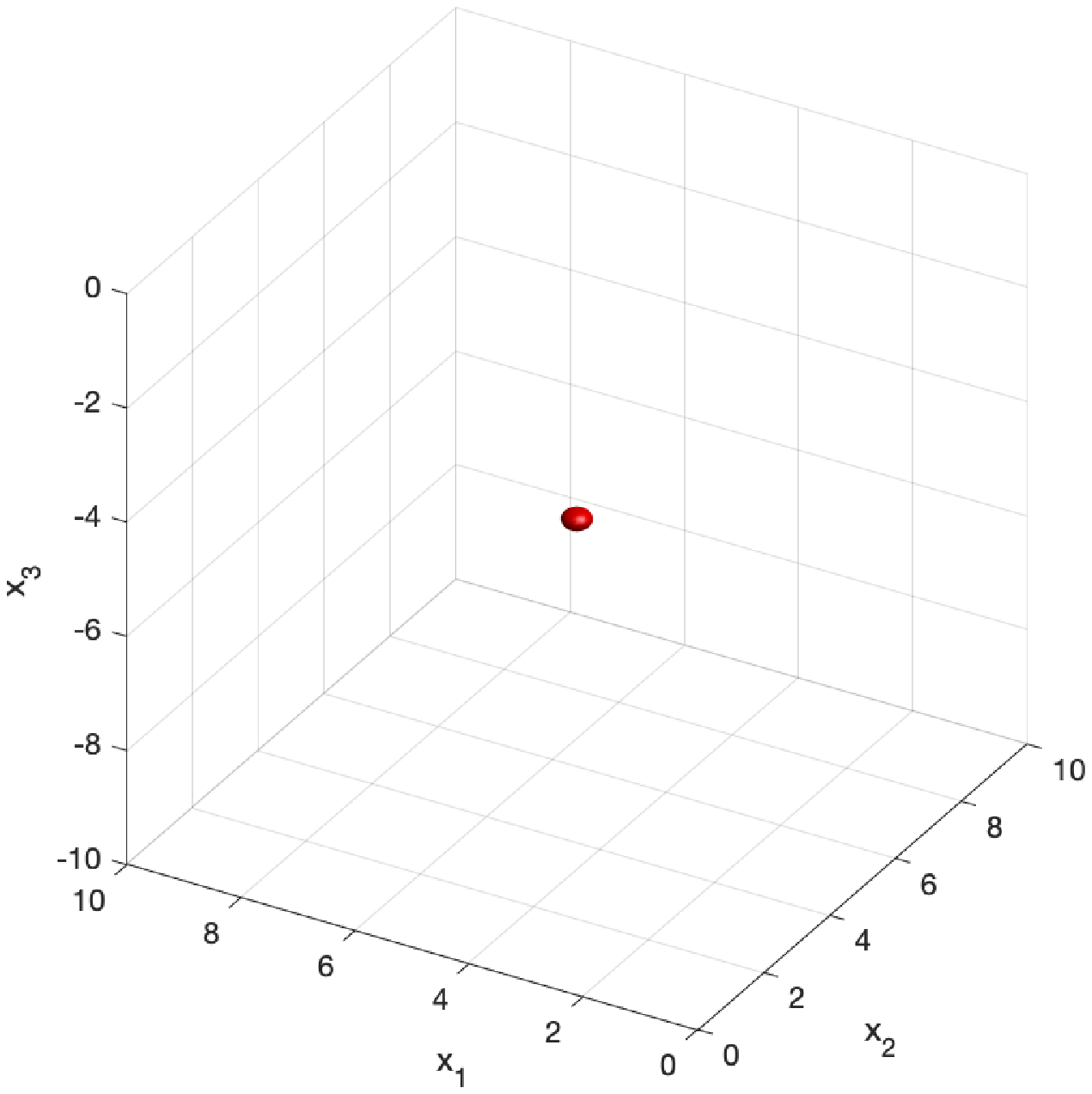}\hspace{0.1cm}\includegraphics[width=0.36\linewidth]{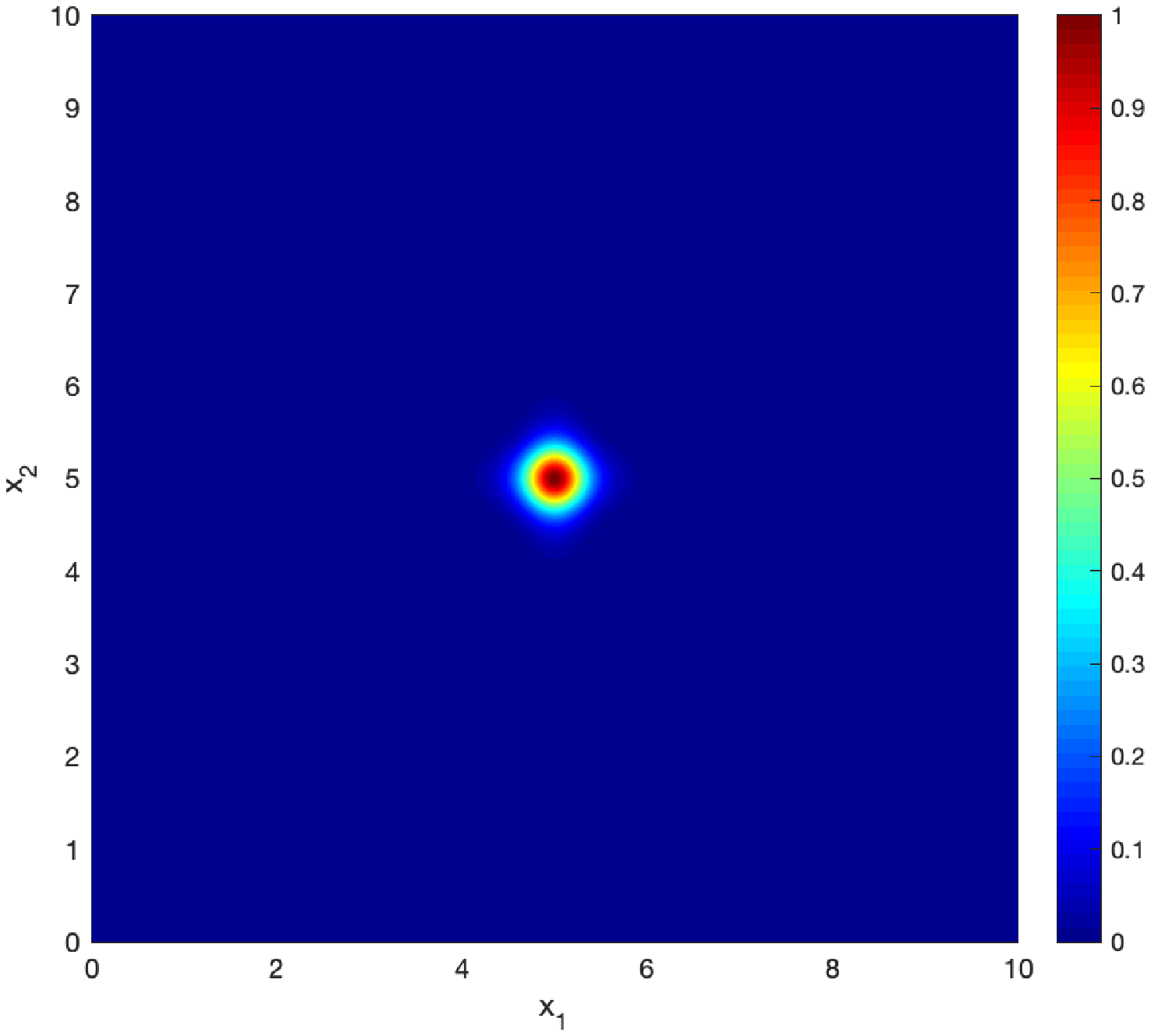}
\includegraphics[width=0.36\linewidth]{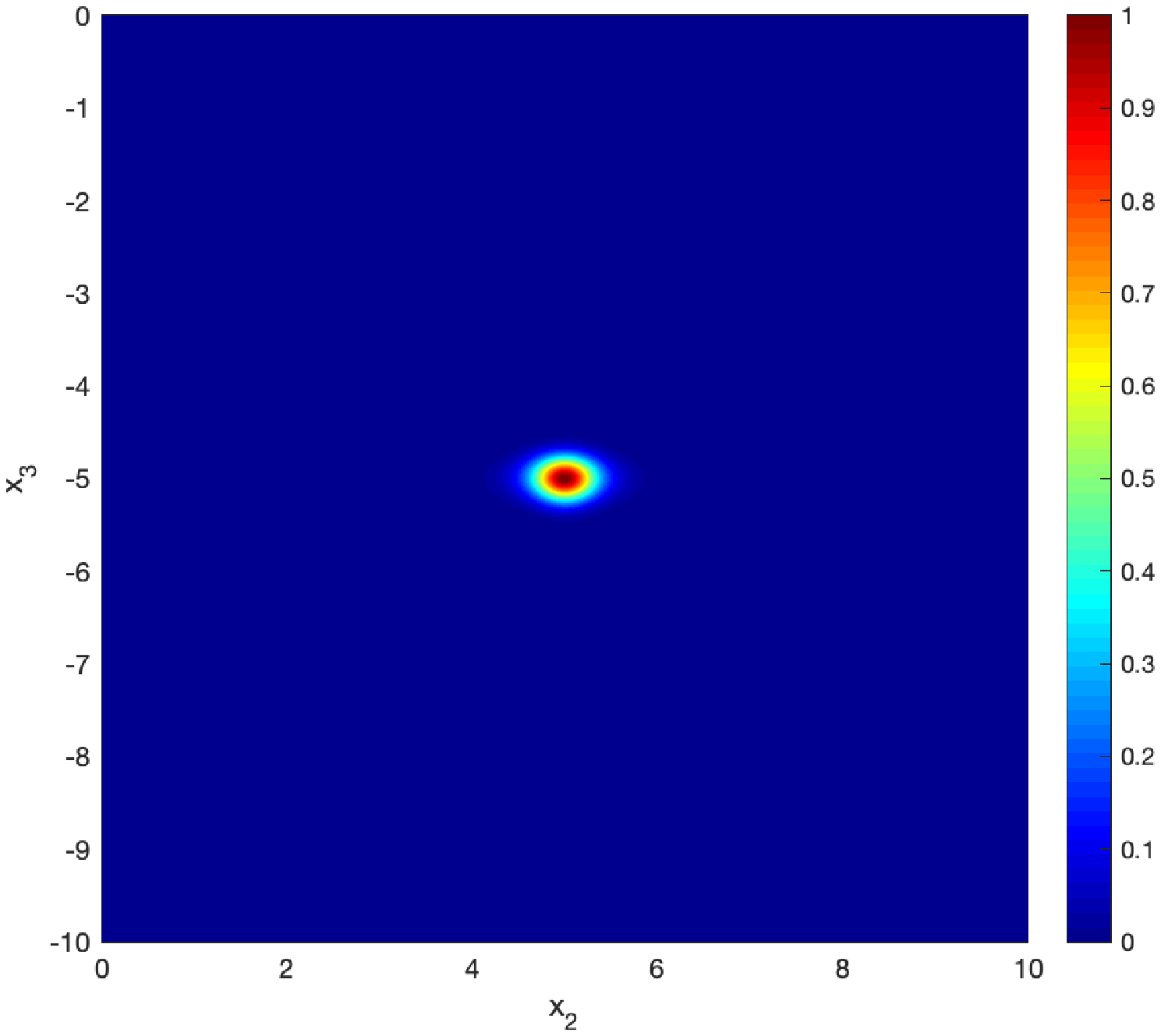}
\includegraphics[width=0.36\linewidth]{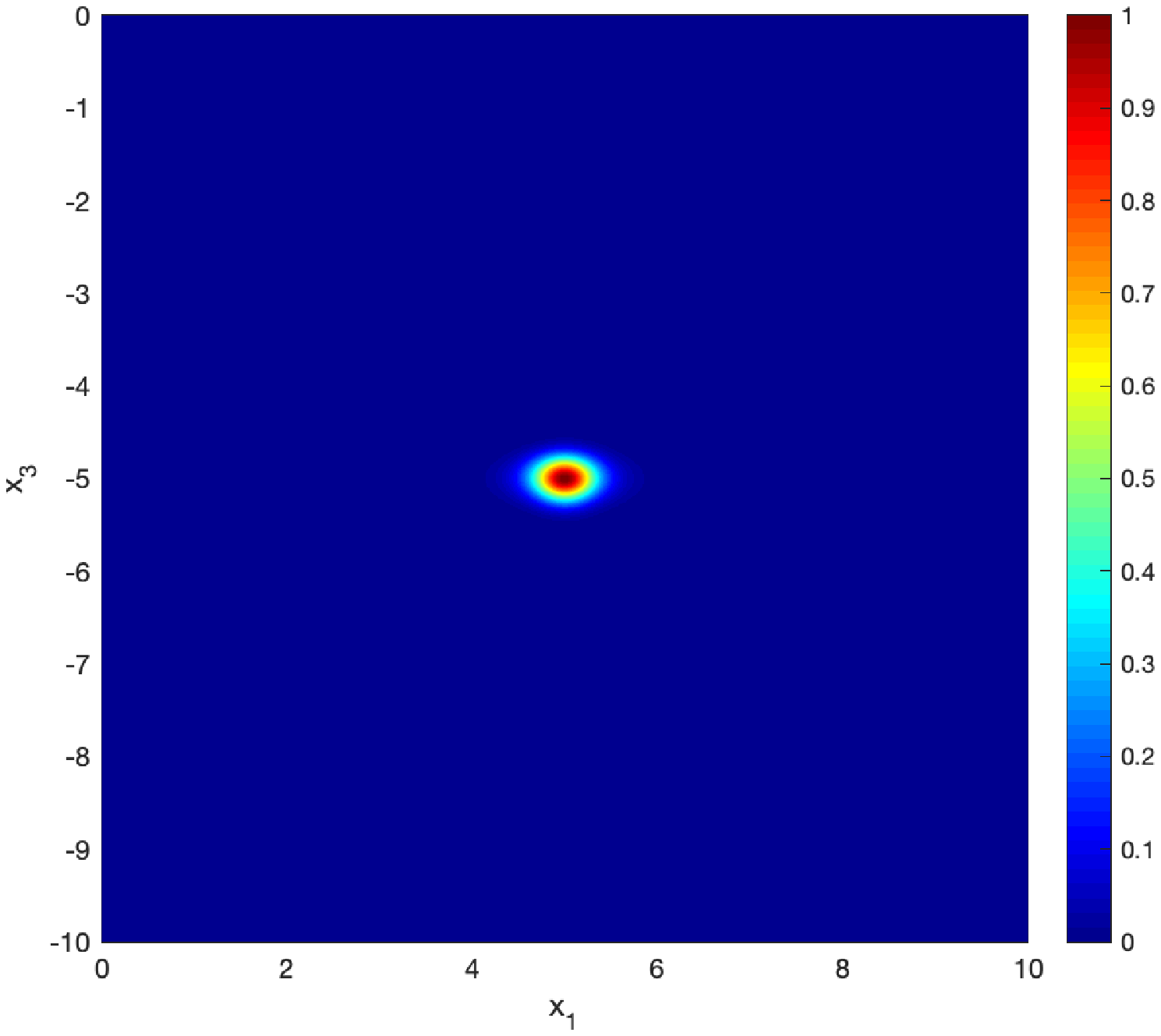}
     \caption{
     \linespread{1}
     Plot of $|\mathcal{H} \Psi (\bx_*;\bz) \be_j|^2$ as a function of the sampling point $\bz$ for a rectangular waveguide with cross section $(0,10)\times(0,10)$ where $\bx_* = (5,5,-5)$ and $k=3$. Top left: Iso-surface plot with iso-value $0.6$. Top right: $x_1x_2$-cross section image. Bottom left: $x_2x_3$-cross section image.  Bottom right: $x_1x_3$-cross section image. 
     } \label{fig point}
    \end{figure}
\begin{remark}
There are conditions to guarantee that the operator $\mathcal{T}$ is coercive. Here we mention one such condition: if $\Im \epsilon$ is bounded below by some positive constant, then $\mathcal{T}$ is coercive. The proof is almost exactly the same as in \cite[Lemma 4.1]{kirsch2004factorization}, we highlight the difference below:  
For any $\vecf \in \vecL^2(D)$,  let $\vecv$ be the unique solution to \eqref{forward scattered e 1} -- \eqref{forward scattered e 2} with $\vecE^i$ replaced by $\vecf$. According to the definition of $\mathcal{T}$ in \eqref{definition operator T}, there is $\mathcal{T}\vecf = k^2(\epsilon-1)(\vecf + \vecv)$. Note that $\epsilon-1$ is compactly supported in $D \subset W_R:=\Sigma \times (-R,0)$, then
\begin{eqnarray*}
\Im \langle \mathcal{T} \vecf, \vecf \rangle &=& k^2 \Im \int_D (\epsilon-1) (\vecf+\vecv) \overline{\vecf} d \bx \\
&=& \Im \Big( k^2 \int_D  (\epsilon-1) |\vecf + \vecv|^2 d \bx - \int_{W_R} |\mcurl \vecv|^2 d \bx \\
&& \qquad + k^2 \int_{W_R} |\vecv|^2 d \bx - \int_{\Sigma_R}  \mcurl \vecv \cdot (\overline{\vecv}\times \nu) d S_\by \Big) \\
&=&  k^2 \Im \int_D  (\epsilon-1) |\vecf + \vecv|^2 d \bx - \Im \int_{\Sigma_R}  (\nu \times \mcurl \vecv) \times \nu \cdot (\overline{\vecv}\times \nu) d S_\by,
\end{eqnarray*}
where the last integral is the duality pairing between $\widetilde{\vecH}^{-1/2}(\mcurl, \Sigma_R)$ and its dual space $\widetilde{\vecH}^{-1/2}(\mdiv, \Sigma_R)$.

From Remark \ref{traces modal}, we can directly write down the modal representation of $\vecv \times \nu|_{\Sigma_R}$ and $(\nu \times \mcurl \vecv) \times \nu|_{\Sigma_R}$ by
\begin{eqnarray*}
\vecv \times \nu|_{\Sigma_R} &=& \sum_{m=1}^\infty a_m \nabla_\Sigma u_m(\widehat{\bx}) - \frac{1}{k}\sum_{n=1}^\infty i g_n b_n \smcurl v_n(\widehat{\bx}), \\
(\nu \times \mcurl \vecv) \times \nu|_{\Sigma_R} &=& \sum_{m=1}^\infty a_m (-ih_m) \nabla_\Sigma u_m(\widehat{\bx}) \\
&&+  \frac{1}{k} \sum_{n=1}^\infty i g_n b_n (ig_n + \frac{-\mu_n^2}{ig_n} ) \smcurl v_n(\widehat{\bx}),
\end{eqnarray*}
for some constants $a_m$ and $b_n$. Therefore 
$
\Im \int_{\Sigma_R}  (\nu \times \mcurl \vecv) \times \nu \cdot (\overline{\vecv}\times \nu) d S_\by \le 0$, and consequently 
$$
\Im \langle \mathcal{T} \vecf, \vecf \rangle \ge  k^2 \Im \int_D  (\epsilon-1) |\vecf + \vecv|^2 d \bx.
$$
Then following \cite[Lemma 4.1]{kirsch2004factorization} we can obtain that $\mathcal{T}$ is coercive.
\end{remark}
    
Our analysis on $|\mathcal{H} \Psi (\bx_*;\bz) \be_j|$ begins with the following lemma.
\begin{lemma} \label{HPsi = green function}
\begin{equation}
 \mathcal{H} \Psi (\bx_*;\bz)\be_j = 
\left\{
\begin{array}{cc}
 \Big( \Re[\partial_{x_{*3}} \tG_e(\bx_*;\bz)] \be_j \Big)_{\small \mbox{prop.}} &  x_{*3}<z_3    \\
     \\
\Big(\Re[\partial_{z_{3}} \tG_e(\bx_*;\bz)] \be_j \Big)_{\small \mbox{prop.}}     &   x_{*3}>z_3
\end{array}
\right.,
\end{equation}
where $``\mbox{prop.}"$ means the projection onto the subspace spanned by the propagating modes.
\end{lemma}
\begin{proof}
We first consider the case $x_{*3}<z_3$. We can directly obtain from \eqref{G modal} that
\begin{eqnarray*}
\partial_{x_{*3}} \tG_e(\bx_*;\bz) &=& 
\sum_{m=1}^\infty c_m(-ih_m)  M_m(\bx_*^-) [M_m^T(\bz) -M_m^T(\bz^-)]   \\
&& \hspace{-3cm}+ \sum_{n=1}^\infty d_n (-ig_n) [ P_n(\bx_*^-) - Q_n(\bx_*^-)] \Big( [ P_n(\bz) - P_n(\bz^-)]^T + [ Q_n(\bz) + Q_n(\bz^-)]^T \Big).
\end{eqnarray*}
We further take the conjugate of the above equation to obtain
\begin{eqnarray*}
\overline{\partial_{x_{*3}} \tG_e(\bx_*;\bz)} 
&=& 
\sum_{m=1}^M c_m(ih_m)  M_m(\bx_*) [M_m^T(\bz) -M_m^T(\bz^-)]  \\
&&  \hspace{-3cm}+\sum_{m=M+1}^\infty c_m(-ih_m)  M_m(\bx_*^-) [M_m^T(\bz) -M_m^T(\bz^-)]   \\
&& \hspace{-3cm}+ \sum_{n=1}^N -d_n (ig_n) [ -P_n(\bx_*) - Q_n(\bx_*)] \Big( [ P_n(\bz) - P_n(\bz^-)]^T + [ Q_n(\bz) + Q_n(\bz^-)]^T \Big) \\
&& \hspace{-3cm} + \sum_{n=N+1}^\infty d_n (-ig_n) [ P_n(\bx_*^-) - Q_n(\bx_*^-)] \Big( [ P_n(\bz) - P_n(\bz^-)]^T + [ Q_n(\bz) + Q_n(\bz^-)]^T \Big).
\end{eqnarray*}
Now we sum up the above two equations to get
\begin{eqnarray*}
2 \Re[\partial_{x_{*3}} \tG_e(\bx_*;\bz)] &=& 
\sum_{m=1}^M c_m(ih_m)  [M_m(\bx_*)-M_m(\bx_*^-)] [M_m^T(\bz) -M_m^T(\bz^-)]   \\
&&\hspace{-3.7cm}+ 2 \sum_{m=M+1}^\infty c_m(-ih_m)  M_m(\bx_*^-) [M_m^T(\bz) -M_m^T(\bz^-)] \\
&& \hspace{-3.7cm}+ \sum_{n=1}^N d_n (ig_n) \Big( [ P_n(\bx_*) - P_n(\bx_*^-)] +  [ Q_n(\bx_*) + Q_n(\bx_*^-)]  \Big)\Big( [ P_n(\bz) - P_n(\bz^-)]^T \\
&&+ [ Q_n(\bz) + Q_n(\bz^-)]^T \Big) \\
&& \hspace{-3.7cm}+ 2\sum_{n=N+1}^\infty d_n (-ig_n) [ P_n(\bx_*^-) - Q_n(\bx_*^-)] \Big( [ P_n(\bz) - P_n(\bz^-)]^T + [ Q_n(\bz) + Q_n(\bz^-)]^T \Big).
\end{eqnarray*}
From the above equation and the expression of $\mathcal{H} \Psi (\bx_*;\bz) \be_j$ in \eqref{HPsi modal}, we obtain that
$$
2 \Big( \Re[\partial_{x_{*3}} \tG_e(\bx_*;\bz)] \be_j \Big)_{\small \mbox{prop.}} = 2\mathcal{H} \Psi (\bx_*;\bz) \be_j.
$$
The case when $x_{*3}>z_3$ can be proved in the same way. This proves the theorem.
\end{proof}
Lemma \ref{HPsi = green function} implies that $|\mathcal{H} \Psi (\bx_*;\bz) \be_j|^2$ peaks when $\bz$ coincide with $\bx_*$ provided there is a good amount of propagating modes. Indeed from the modal representation in \eqref{HPsi modal}, we can directly plot  $|\mathcal{H} \Psi (\bx_*;\bz) \be_j|^2$
as a function of $\bz$ for any fixed $\bx_*$. In Fig.~\ref{fig point}, we plot $|\mathcal{H} \Psi (\bx_*;\bz) \be_j|^2$ for a rectangular waveguide with cross section $(0,10)\times(0,10)$ where $\bx_* = (5,5,-5)$ and $k=3$. We observe that $|\mathcal{H} \Psi (\bx_*;\bz) \be_j|^2$ peaks when $\bz$ coincide with $\bx_*$. 

Let us summarize this section. Theorem \ref{resolution I_j(z)} implies that that $I_j(\bz)$ behaves qualitatively as $\|  \mathcal{H} \Psi(\cdot;\bz) \be_j\|^2_{\vecL^2(D)}$. Since  $|\mathcal{H} \Psi (\bx_*;\bz) \be_j|$ peaks when $\bz$ coincide with $\bx_*$ as evidenced by Lemma \ref{HPsi = green function} and Fig.~\ref{fig point}, thus we can conclude that $\|  \mathcal{H} \Psi(\cdot;\bz) \be_j\|^2_{\vecL^2(D)}$ {(where we can approximate the $\vecL^2(D)$ norm by quadrature rules) peaks} in the scatterer $D$, and so does  $I(\bz)$.
\begin{remark}
If the scatterer $D$ is known as a priori, we can image $D$ by $\|  \mathcal{H} \Psi(\cdot;\bz) \be_j\|^2_{\vecL^2(D)}$; unfortunately, $D$ is what we aim to image. The robustness and efficiency of our imaging function $I(\bz)$ is rooted in that it behaves qualitatively as $\|  \mathcal{H} \Psi(\cdot;\bz) \be_j\|^2_{\vecL^2(D)}$ without knowing $D$ as a priori.
\end{remark}
\subsection{Implementation of Imaging Function}
In this section, we give the modal representation of the imaging function $I(\bz)$ in \eqref{imaging function Iz}.

To begin with, let us introduce
\begin{equation}
\vecU^s(\bx;\by):=(  \vecu^s(\bx;\by;\be_1), \vecu^s(\bx;\by;\be_2), \vecu^s(\bx;\by;\be_3) )
\end{equation}
for all $\bx \in \Sigma_r$ and all $\by \in \Sigma_r$. Now we define a $(M+N)\times(M+N)$-dimensional matrix $\vecU^s$ with $jj'$ entry $\vecU^s_{j j'}$ given by 
{ \tiny
\begin{equation} \label{Implementation matrix Us}
\left\{
\begin{array}{c}
\hspace{-5.7cm}\frac{1}{\lambda_j^2 \lambda_{j'}^2}\int\int \overline{M}^T_{j'}(\bx^-) \vecU^s(\bx;\by) \overline{M}_j(\by^-) \,d S_\by \,d S_\bx,    j \le M,  j' \le M    \\
\hspace{-2.5cm}\frac{1}{\lambda_{j}^2 \mu_{j'-M}^2} \int\int [\overline{P}_{j'-M}(\bx^-)- \overline{Q}_{j'-M}(\bx^-)  ]^T \vecU^s(\bx;\by) \overline{M}_j(\by^-) \,d S_\by \,d S_\bx,    j \le M,  j' \ge M    +1\\
\hspace{-3.35cm}\frac{1}{\mu_{j-M}^2 \lambda_{j'}^2} \int\int \overline{M}^T_{j'}(\bx) \vecU^s(\bx;\by) [\overline{P}_{j-M}(\bx^-)- \overline{Q}_{j-M}(\bx^-)  ] \,d S_\by \,d S_\bx,   j \ge M,  j' \le M    \\
\hspace{-0.25cm}\frac{1}{\mu_{j-M}^2 \mu_{j'-M}^2} \int\int [\overline{P}_{j'-M}(\bx^-)- \overline{Q}_{j'-M}(\bx^-)  ]^T \vecU^s(\bx;\by) [\overline{P}_{j-M}(\bx^-)-\overline{Q}_{j-M}(\bx^-)  ] \,d S_\by \,d S_\bx,   j \ge M,  j' \ge M  
\end{array}
\right.
\end{equation}
}
With such notation, we have that $I_\ell(\bz)$ in \eqref{imaging function Iz} (for $\ell=1,2,3$) can be written as
\begin{eqnarray*}
I_\ell(\bz) &=& \big|\langle \overline{\mathcal{N} \Psi(\cdot;\bz) \be_\ell} \times \nu,   \Psi(\cdot;\bz) \be_\ell\rangle \big| \\
&=& \Big| \int\int [\Psi(\bx;\bz) \be_\ell]^T \vecu^s(\bx;\by;\Psi(\by;\bz) \be_\ell)  \,d S_\by \, d S_\bx \Big| \\
&=& \Big| \int\int [\Psi(\bx;\bz) \be_\ell]^T \vecU^s(\bx;\by)  [\Psi(\by;\bz) \be_\ell]  \,d S_\by \, d S_\bx \Big| =\sum_{j=1}^{M+N} \sum_{j'=1}^{M+N} \vecg^\ell_{j'} \vecU{jj'} \vecg^\ell_j \\
&=& (\vecg^\ell)^T \vecU \vecg^\ell
\end{eqnarray*}
where $\vecg^\ell$ is a $(M+N)\times 1$ vector with $j$-th entry given by
\begin{equation} \label{Implementation vector g}
\vecg^\ell_{j}:=
\left\{
\begin{array}{cc}
\int  M^T_j(\by^-)[\Psi(\by;\bz) \be_\ell]  \,d S_\by,  &  j \le M   \\
\\
 \int    [P_{j-M}(\by^-) - Q_{j-M}(\by^-)]^T  [\Psi(\by;\bz) \be_\ell] \,d S_\by,  &  j \ge M+1
\end{array}
\right.
\end{equation}
With $\vecU$ given by \eqref{Implementation matrix Us} and $\vecg^\ell$ given by \eqref{Implementation vector g}, the imaging function $I(\bz)$ can be written as
\begin{equation} \label{imaging function modal}
I(\bz) = \sum_{\ell=1}^3 \big|(\vecg^\ell)^T \vecU \vecg^\ell \big|.
\end{equation}
    \begin{figure}[ht!]
    \centering
\includegraphics[width=0.422\linewidth]{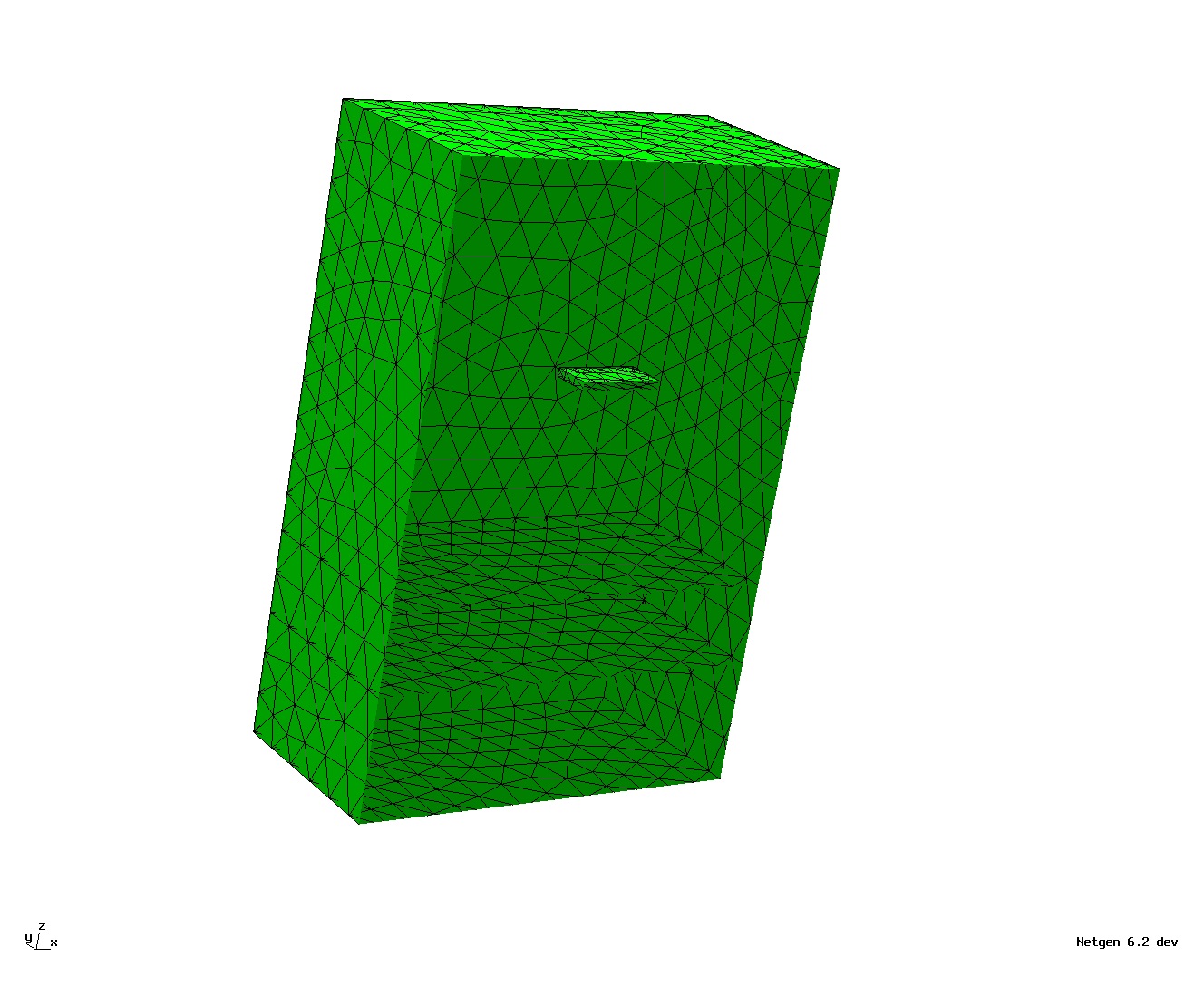}
     \caption{
     \linespread{1}
     Three dimensional view of the generated mesh.
     } \label{domain}
    \end{figure}
\section{Numerical Examples}\label{Numerics}
In this section, we provide numerical examples to illustrate our sampling type method.

    \begin{figure}[ht!]
    \includegraphics[width=0.33\linewidth]{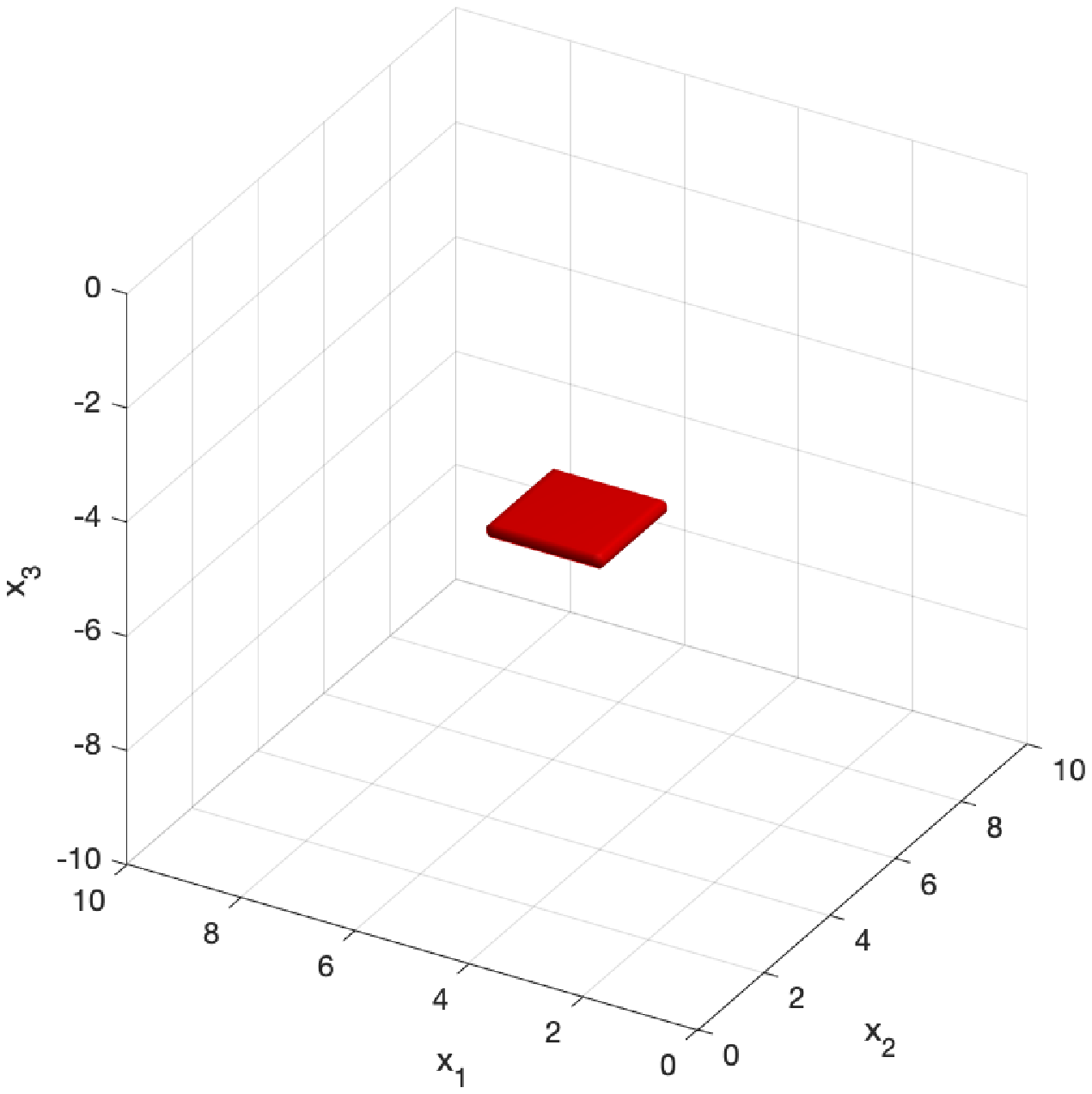}\includegraphics[width=0.33\linewidth]{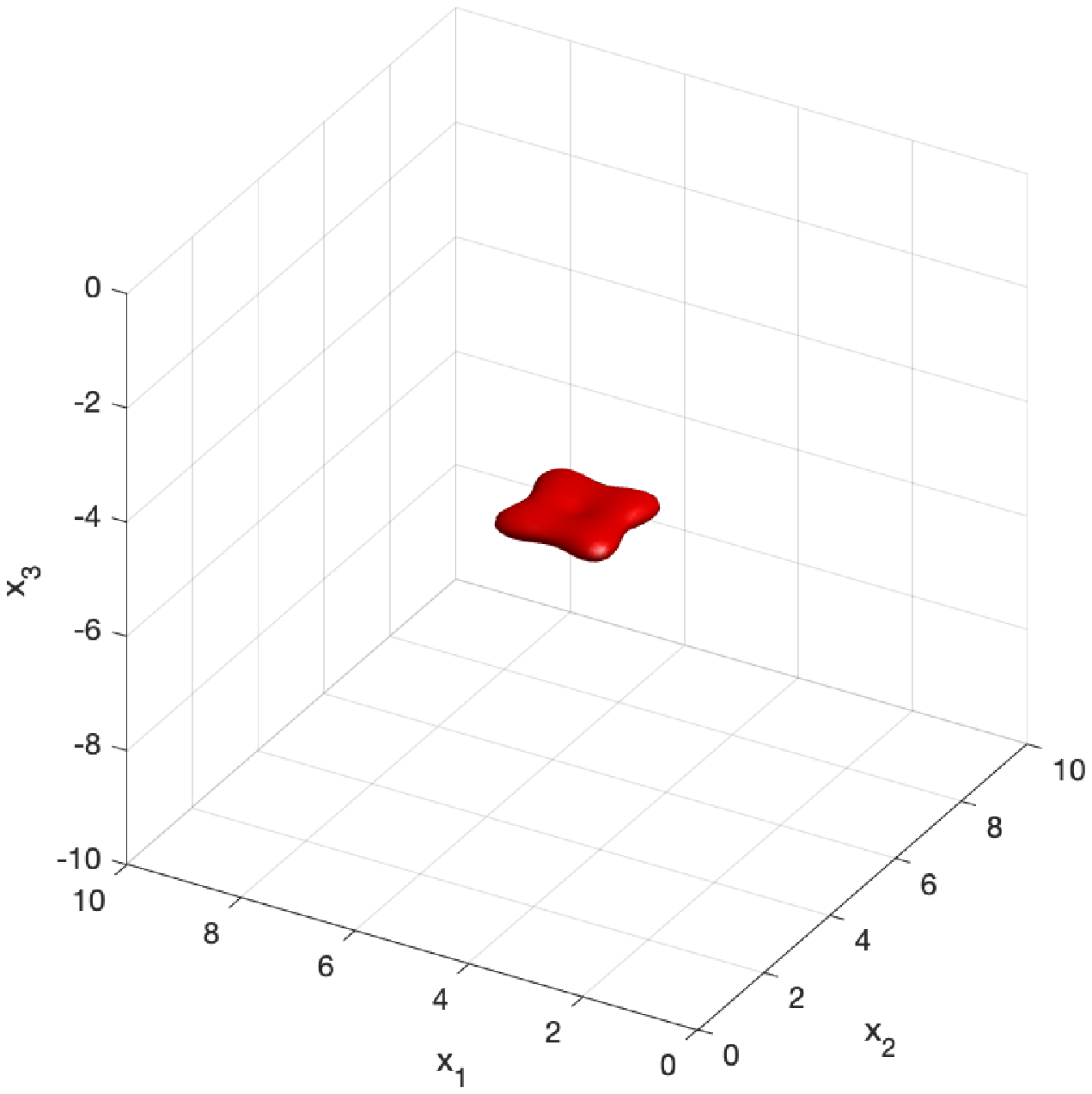}\includegraphics[width=0.33\linewidth]{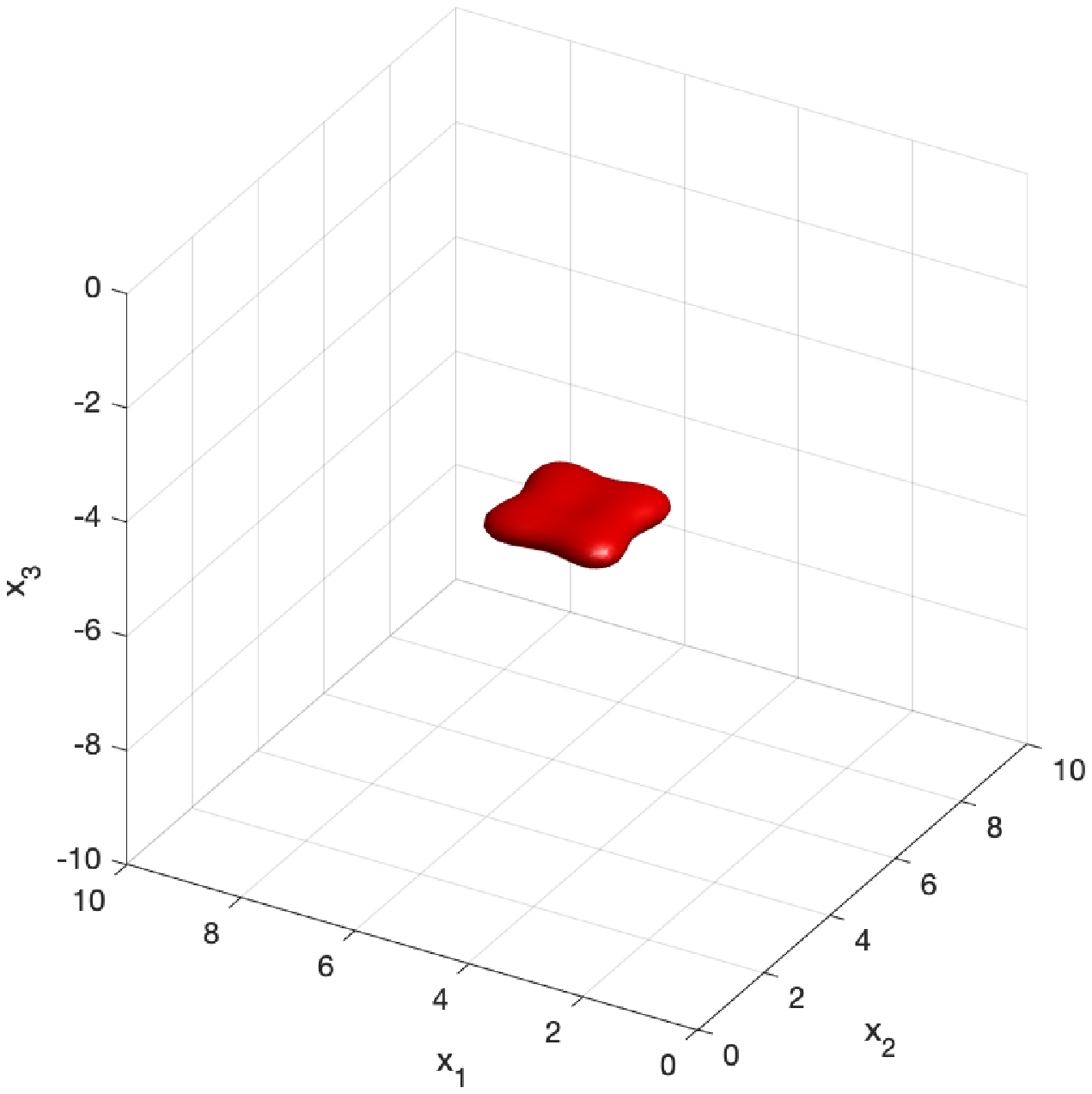}
\includegraphics[width=0.33\linewidth]{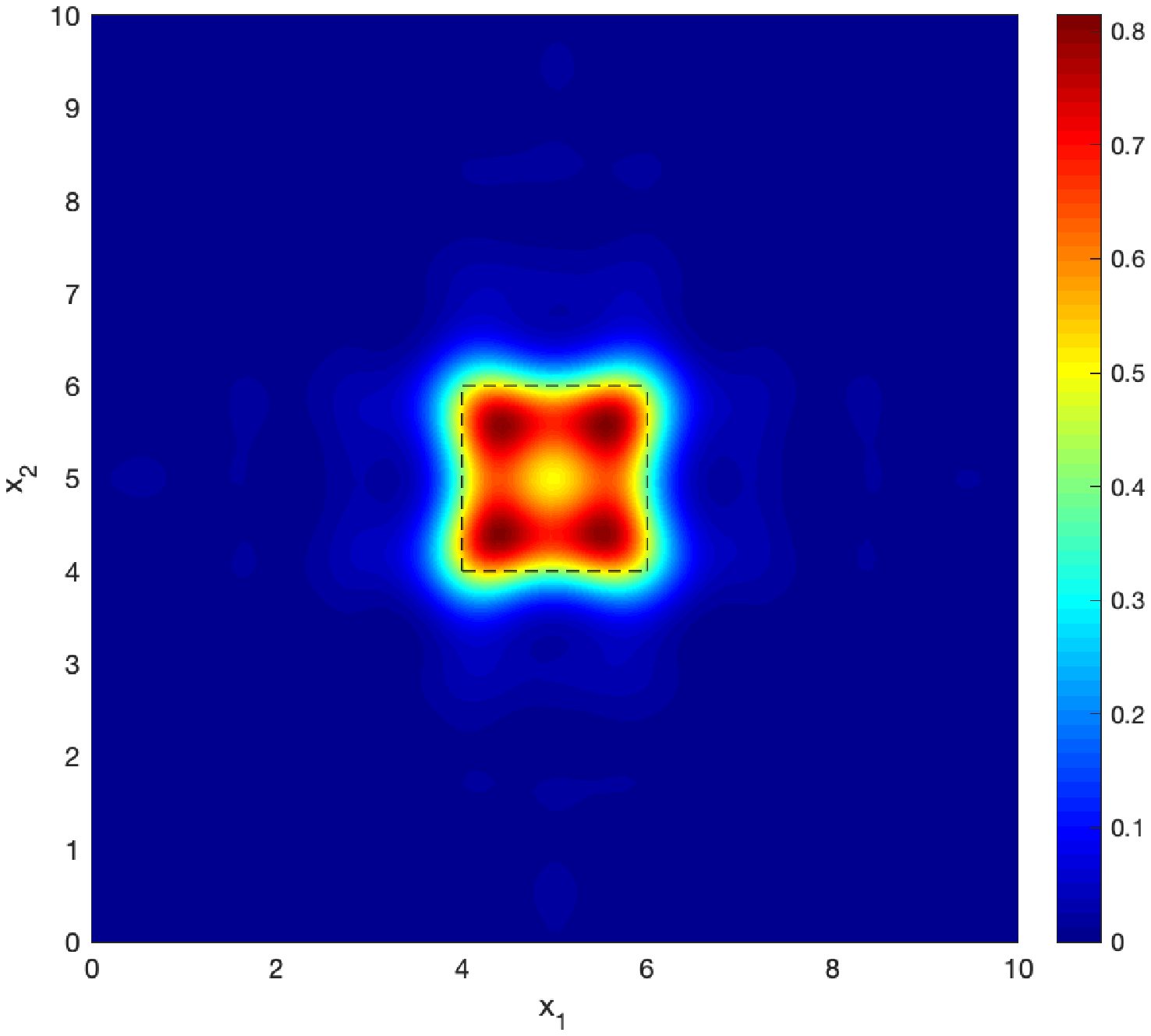}\includegraphics[width=0.33\linewidth]{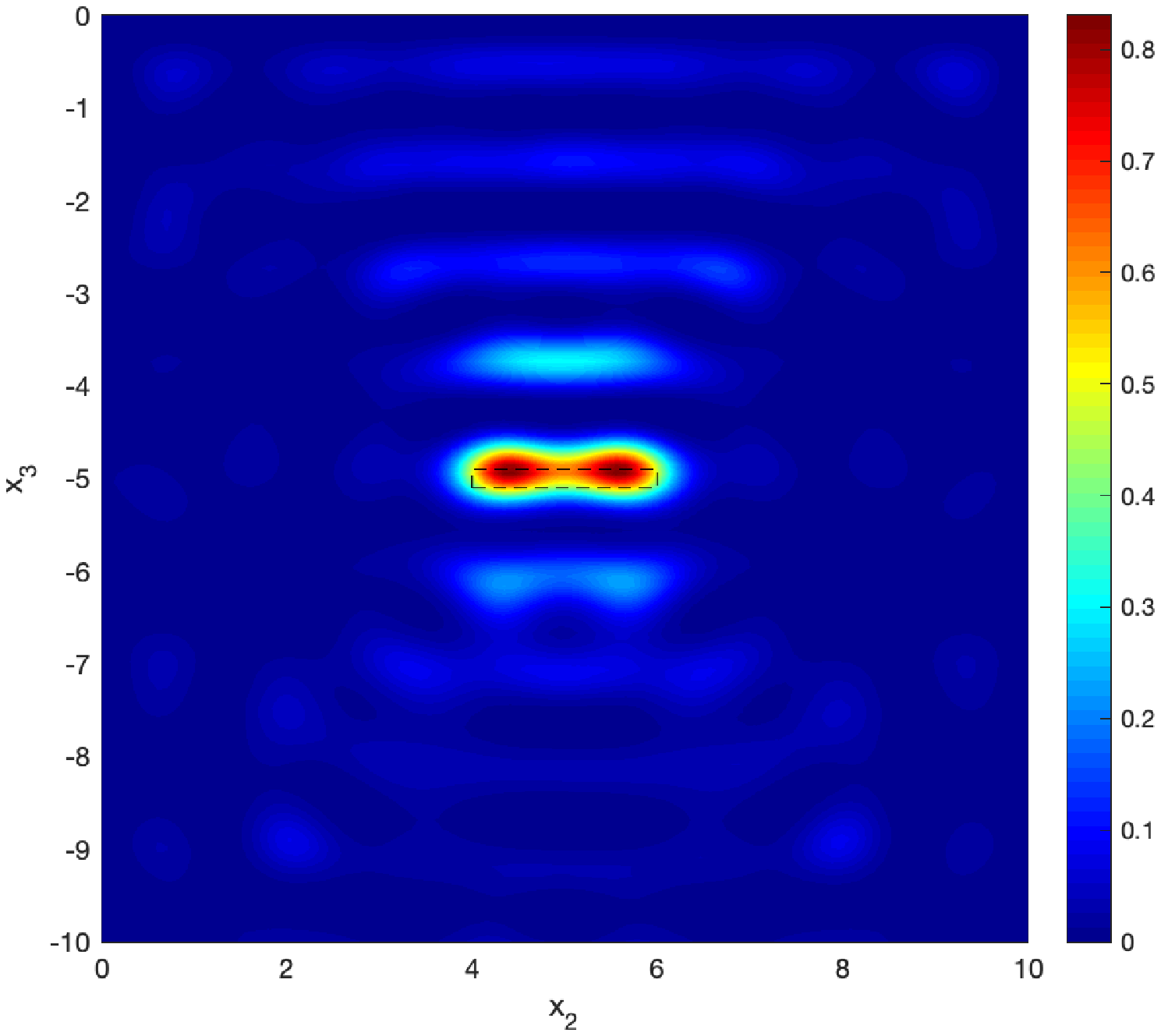}\includegraphics[width=0.33\linewidth]{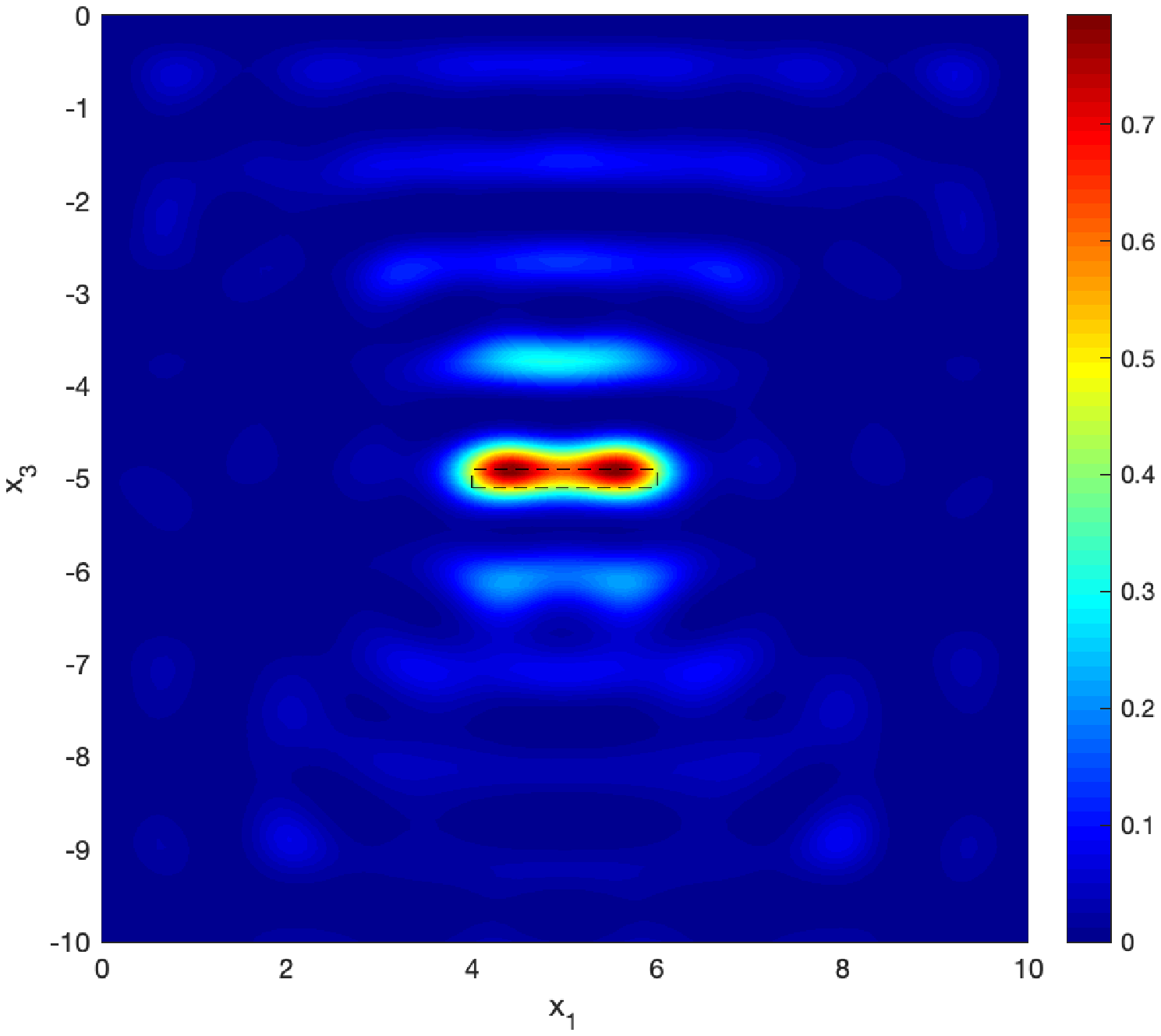}
     \caption{
     \linespread{1}
Image of a {cuboid}. Top left: exact. Top middle: three dimensional image using iso-surface plot with iso-value $0.6$. Top right: three dimensional image using iso-surface plot with iso-value $0.4$. 
At the bottom, we plot the cross section images where the exact geometry is indicated by the dashed line. Bottom left: $x_1x_2$-cross section image.  Bottom middle: $x_2x_3$-cross section image. Bottom right: $x_1x_3$-cross section image. 
     } \label{fig c1}
    \end{figure}

We consider a rectangular waveguide with cross-section $\Sigma = (0,10)\times(0,10)$. We generate the synthetic data $\vecU$ using the Finite Element computational software Netgen/NGSolve \cite{schoberl1997netgen}. To be more precise, the computational domain is $\Sigma \times (-15,0)$ and the measurements are on the cross section $ \Sigma \times \{-10\}$. We apply a Perfectly Matched Layer (PML) in $\Sigma \times (-15,-12)$ and choose a complex-valued PML absorbing coefficient $8+8i$ to handle both the propagating modes and evanescent modes. See Fig. \ref{domain} for an illustration. We directly compute $\vecU$ using Netgen/NGSolve ``Integrate" function after we have  computed the measurements $\vecu^s(\bx;\by;\be_\ell)$ with $\bx,\by$ on the measurement surface and $\ell=1,2,3$. We further add $5\%$ Gaussian noise to the synthetic data $\vecU$ to implement the imaging function given by \eqref{imaging function modal} in Matlab; for the best visualization, we plot $I^2(\bz)$ and we always normalize it such that the maximum value is $1$, unless otherwise specified.

In all of the numerical examples, we set the relative electric permittivity $\epsilon=2+2i$ and apply the quadratic edge element to solve for the scattered electric wave field. The mesh size is chosen as $0.9$ in the homogeneous waveguide domain and $0.9/\sqrt{2}$ in the scatterer domain. In most examples, the wavenumber is taken as $k=3$ (unless specified), in which case there are $146$ propagating modes { ($M=82$ and $N=64$). We further consider smaller and larger wavenumbers to illustrate the performance of our sampling type method with respect to different wavenumbers (see Fig. \ref{fig lcyl} and \ref{fig lball}).}

    \begin{figure}[ht!]
    \includegraphics[width=0.33\linewidth]{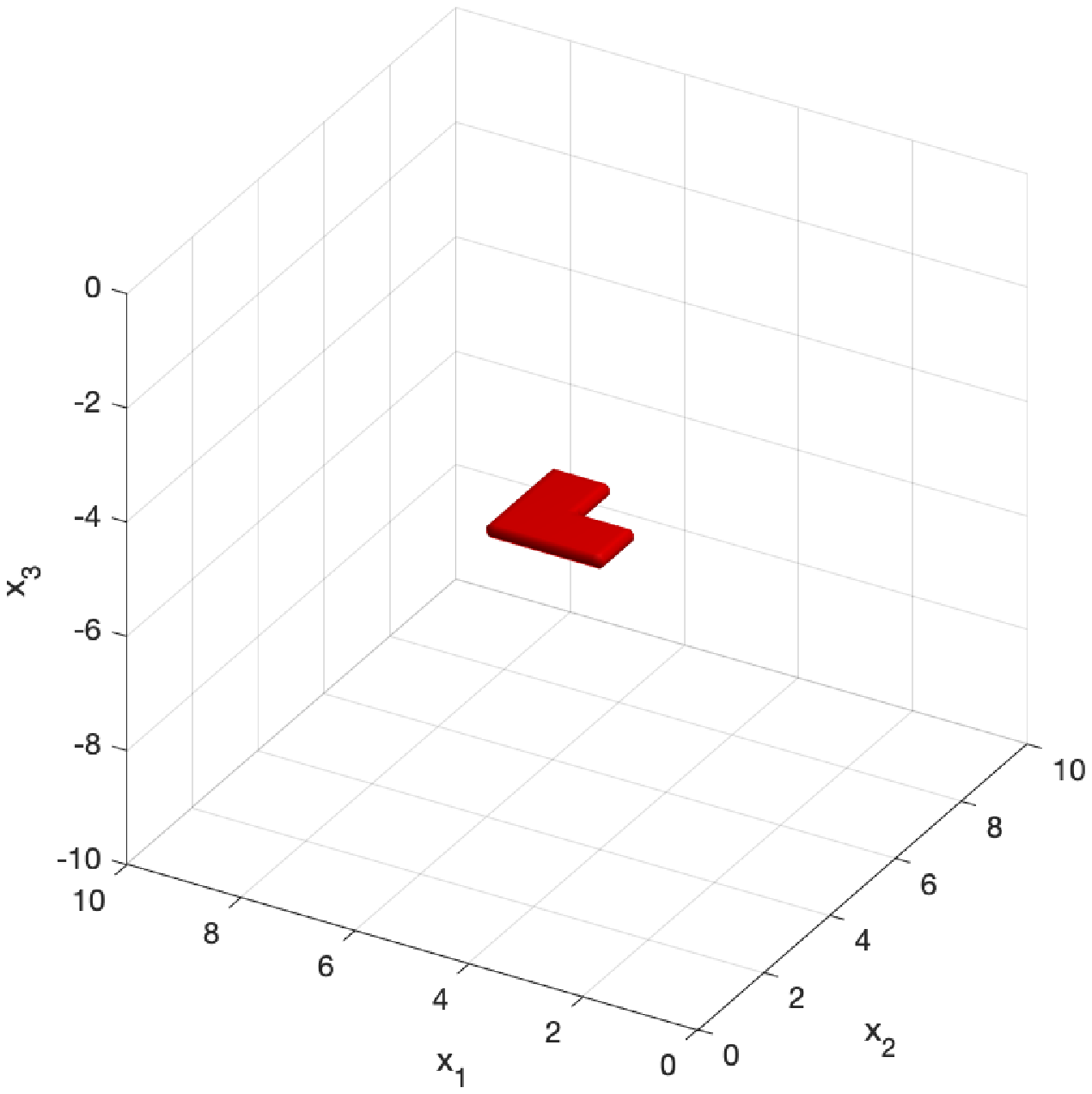}\includegraphics[width=0.33\linewidth]{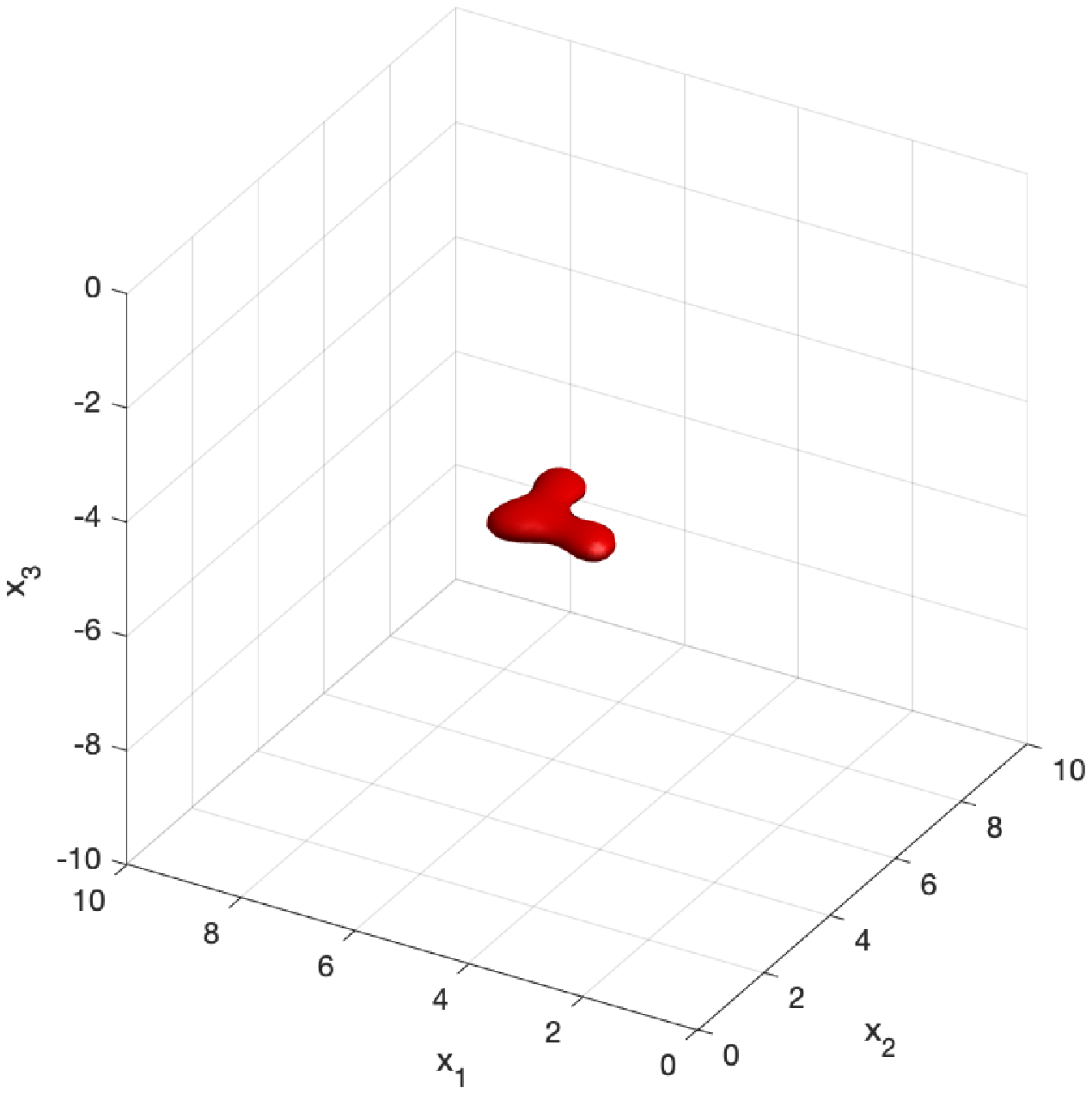}\includegraphics[width=0.33\linewidth]{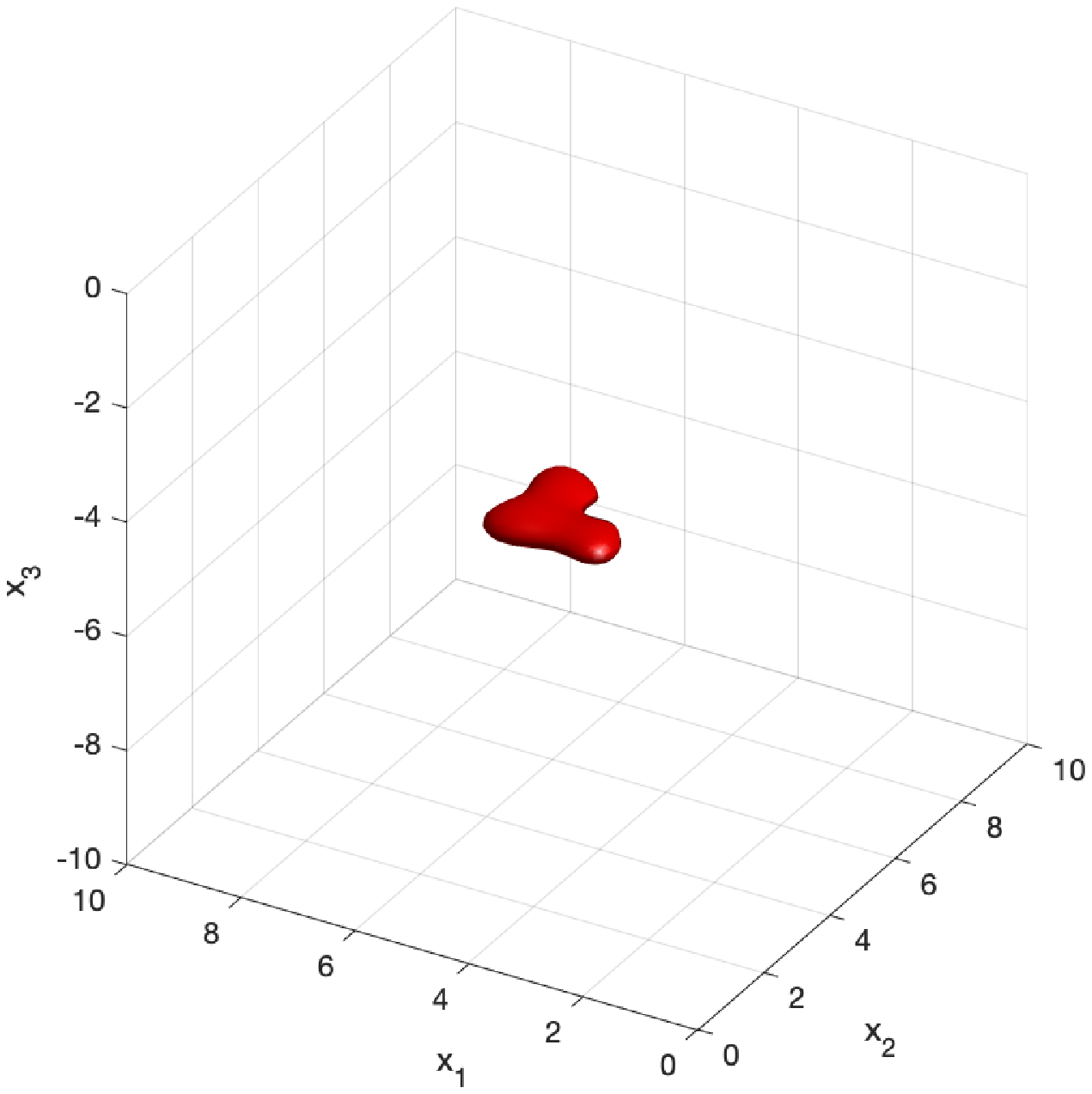}
\includegraphics[width=0.33\linewidth]{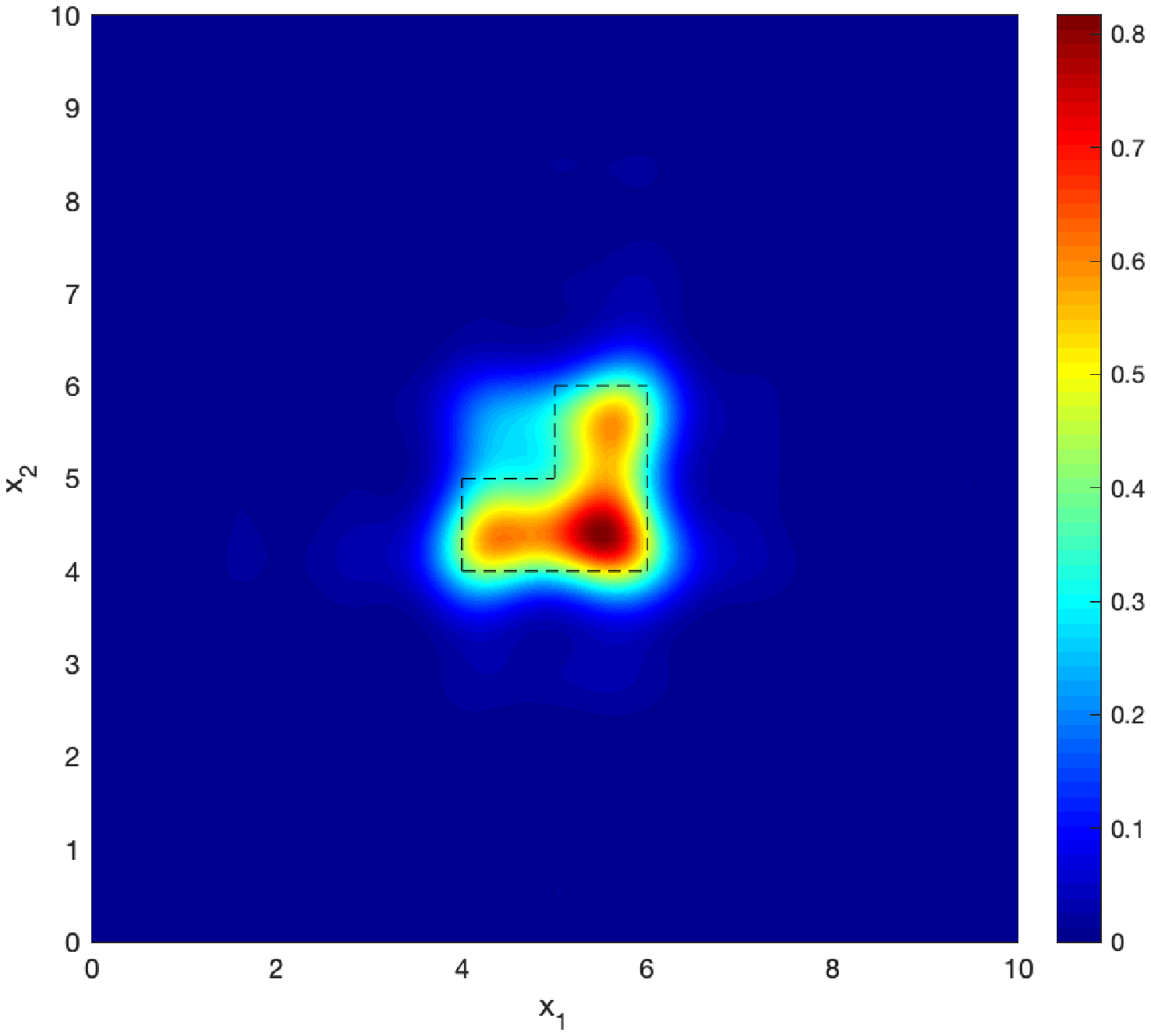}\includegraphics[width=0.33\linewidth]{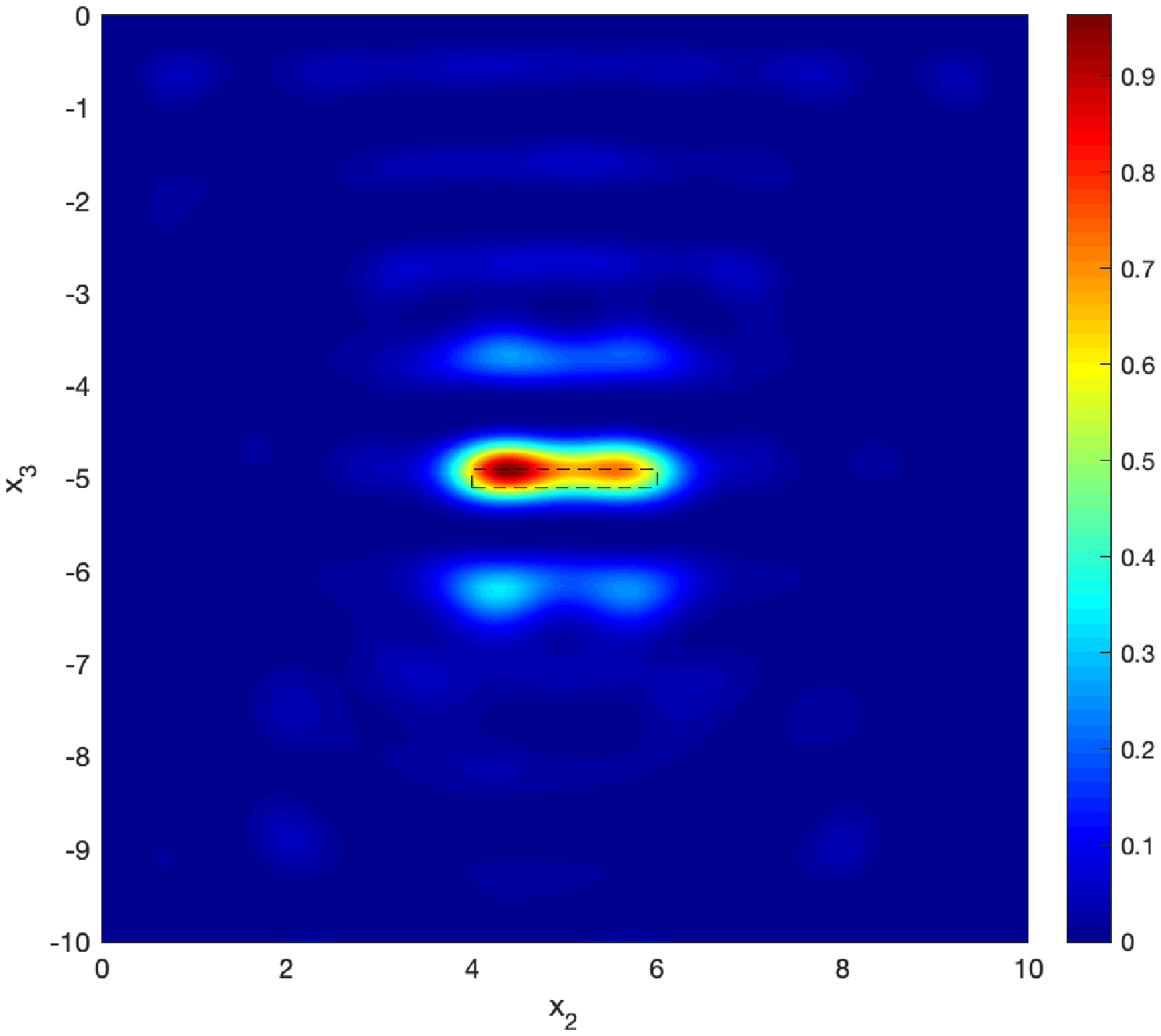}\includegraphics[width=0.33\linewidth]{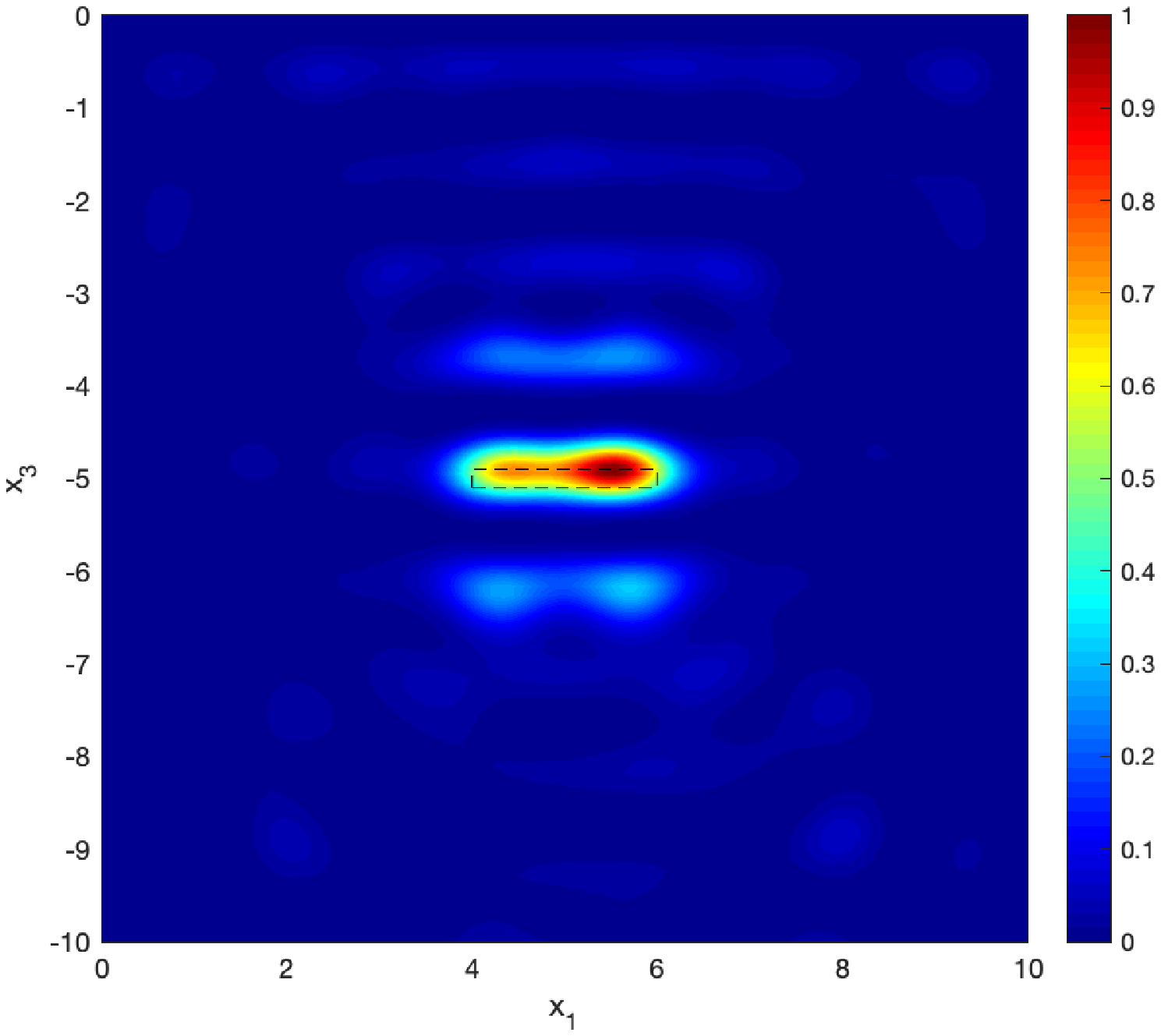}
     \caption{
     \linespread{1}
Image of a L-shape scatterer. Top left: exact. Top middle: three dimensional image using iso-surface plot with iso-value $0.5$. Top right: three dimensional image using iso-surface plot with iso-value $0.4$.
At the bottom, we plot the cross section images where the exact geometry is indicated by the dashed line. Bottom left: $x_1x_2$-cross section image.  Bottom middle: $x_2x_3$-cross section image. Bottom right: $x_1x_3$-cross section image. 
     } \label{fig L}
    \end{figure}
    
The first numerical example Fig.~\ref{fig c1} is imaging of a {cuboid} given by
\begin{eqnarray*}
\{ \bx: |x_1-5|< 1, |x_2-5|< 1, |x_3+5|< 0.1 \}.
\end{eqnarray*}
We provide both three-dimensional image and cross section images. The two three-dimensional images are generated using Matlab function ``isosurface" with iso-value $0.6$ and $0.4$ respectively. The $x_1x_2$-cross section image is at $x_3=-5$,  the $x_2x_3$-cross section image is at $x_1=5$, and the $x_1x_3$-cross section image is at $x_2=5$.

The next numerical example Fig.~\ref{fig L} is imaging of a L-shape scatterer given by
\begin{eqnarray*}
\{\bx: -1<x_1-5<0, -1<x_2-5<0, |x_3+5|<0.1\} \\\cup \{\bx: 0<x_1-5<1, |x_2-5|<1, |x_3+5|<0.1\}.
\end{eqnarray*}
Here we plot the three dimensional images using iso-surface plot with iso-value $0.5$ and $0.4$ respectively. 

{
To further illustrate our imaging method, we consider an elongated scatterer, a cylindrical scatterer given by
\begin{eqnarray*}
\{ \bx: |x_1-5|^2 +  |x_2-5|^2< 0.1^2, |x_3+5|< 1 \}.
\end{eqnarray*}
Fig. \ref{fig lcyl} provides images with $k=3$ ($M=82$ and $N=64$) and $k=5$ ($M=213$ and $N=183$). We first observe  that one can obtain better images using higher wavenumber. Furthermore, together with Fig. \ref{fig c1} -- \ref{fig L}, we can conclude that our imaging method can image both flat and elongated scatterers. 
}
    \begin{figure}[ht!]
    \includegraphics[width=0.33\linewidth]{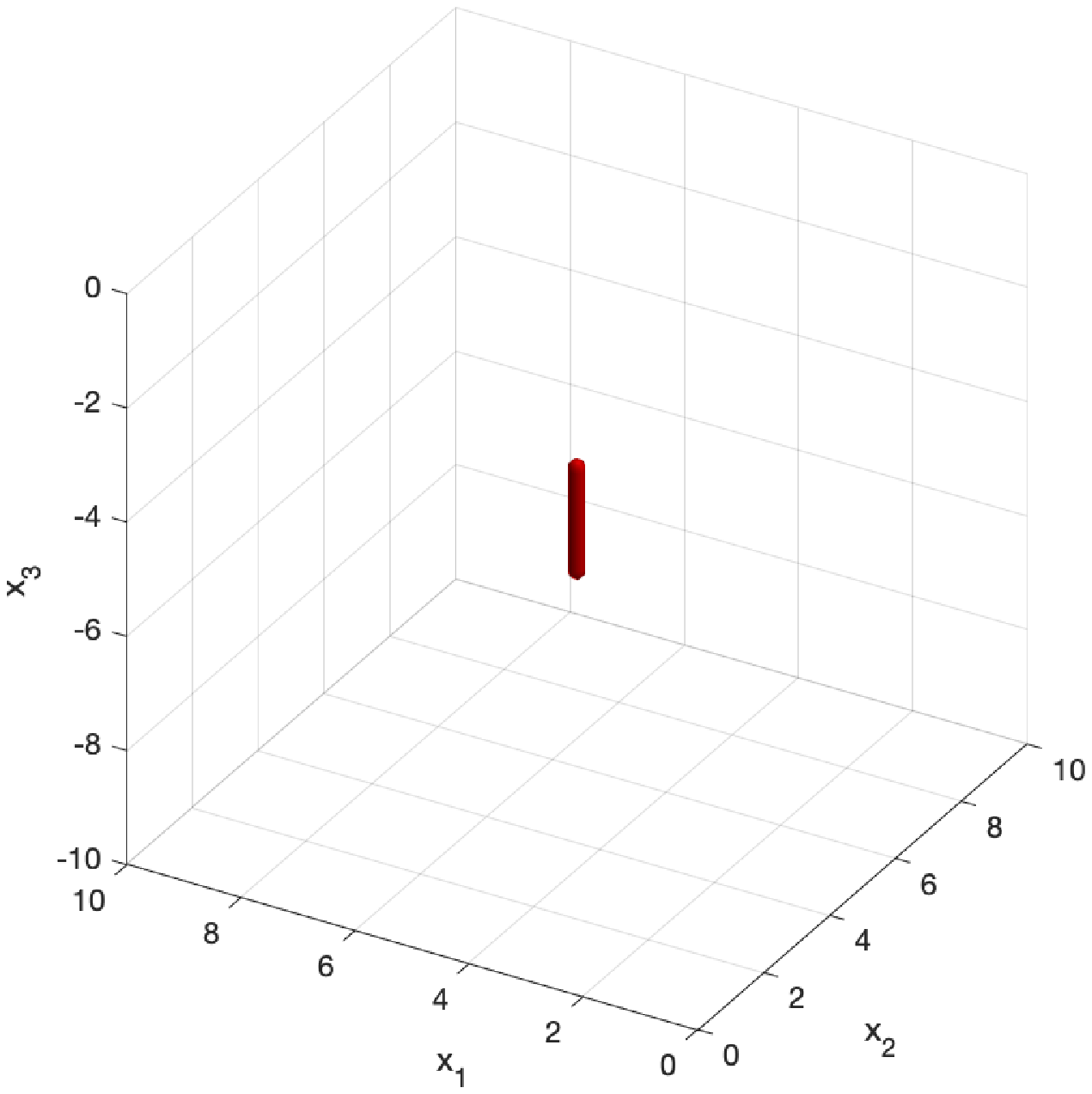}\includegraphics[width=0.33\linewidth]{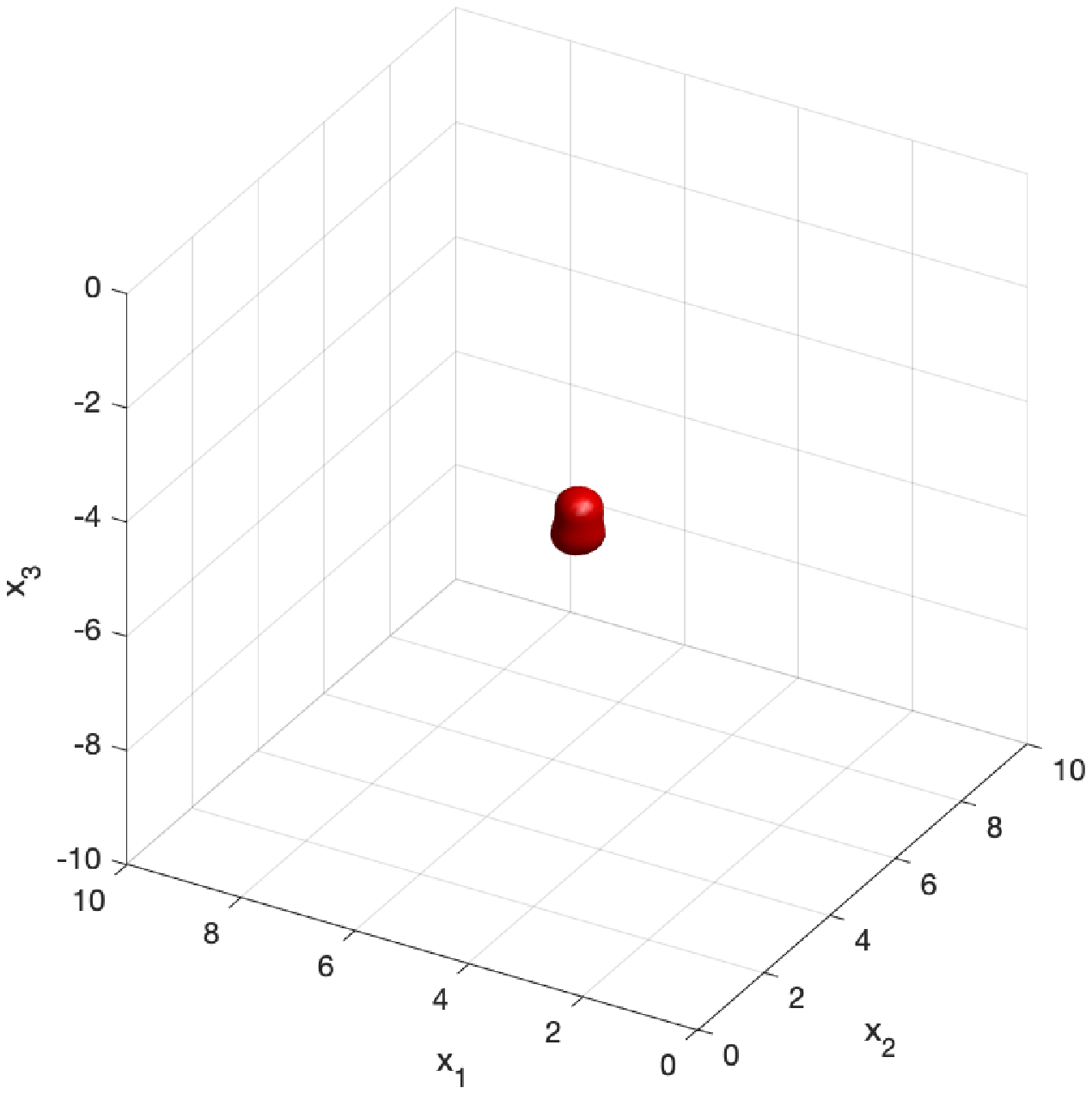}\includegraphics[width=0.33\linewidth]{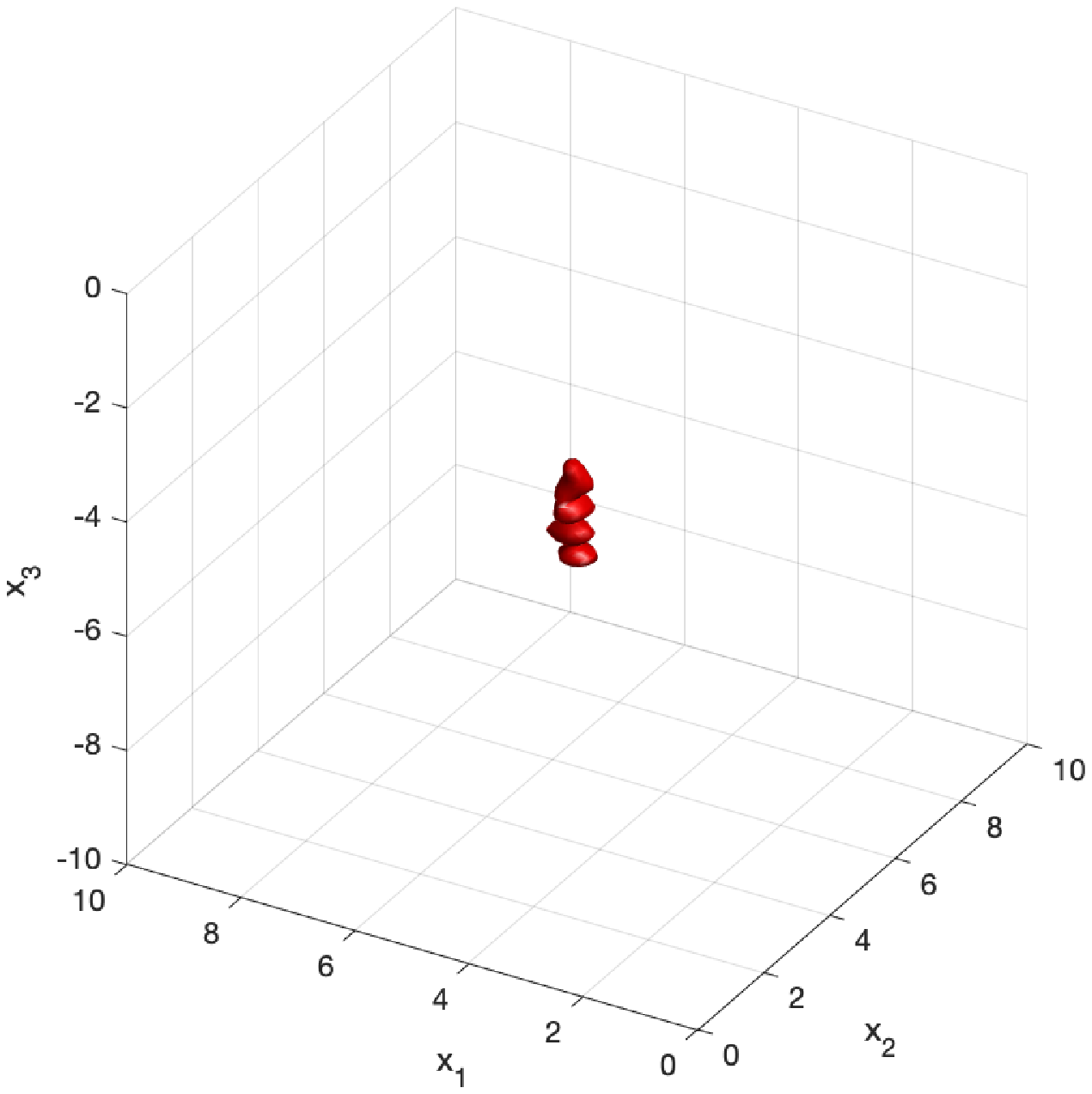}
\includegraphics[width=0.33\linewidth]{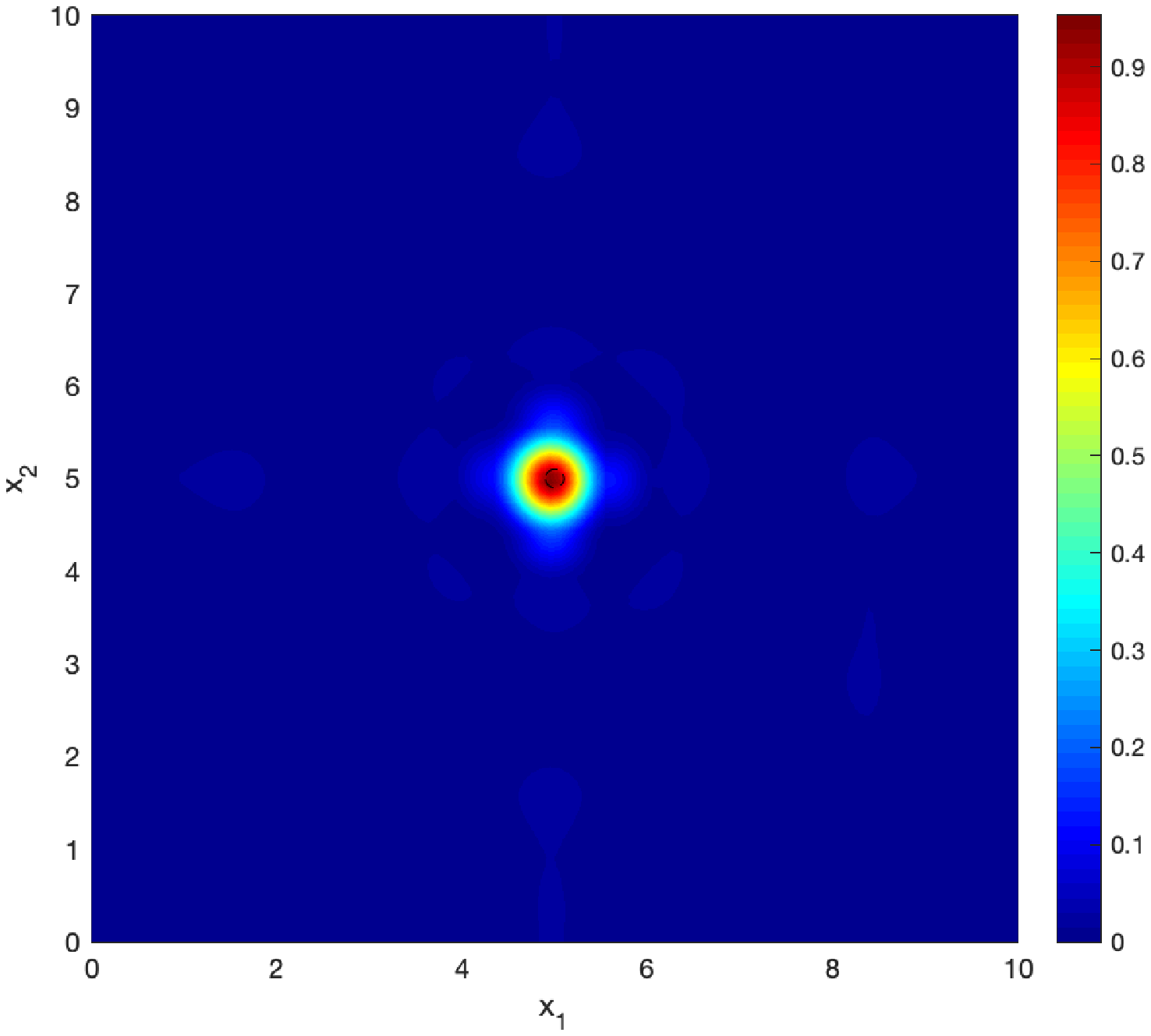}\includegraphics[width=0.33\linewidth]{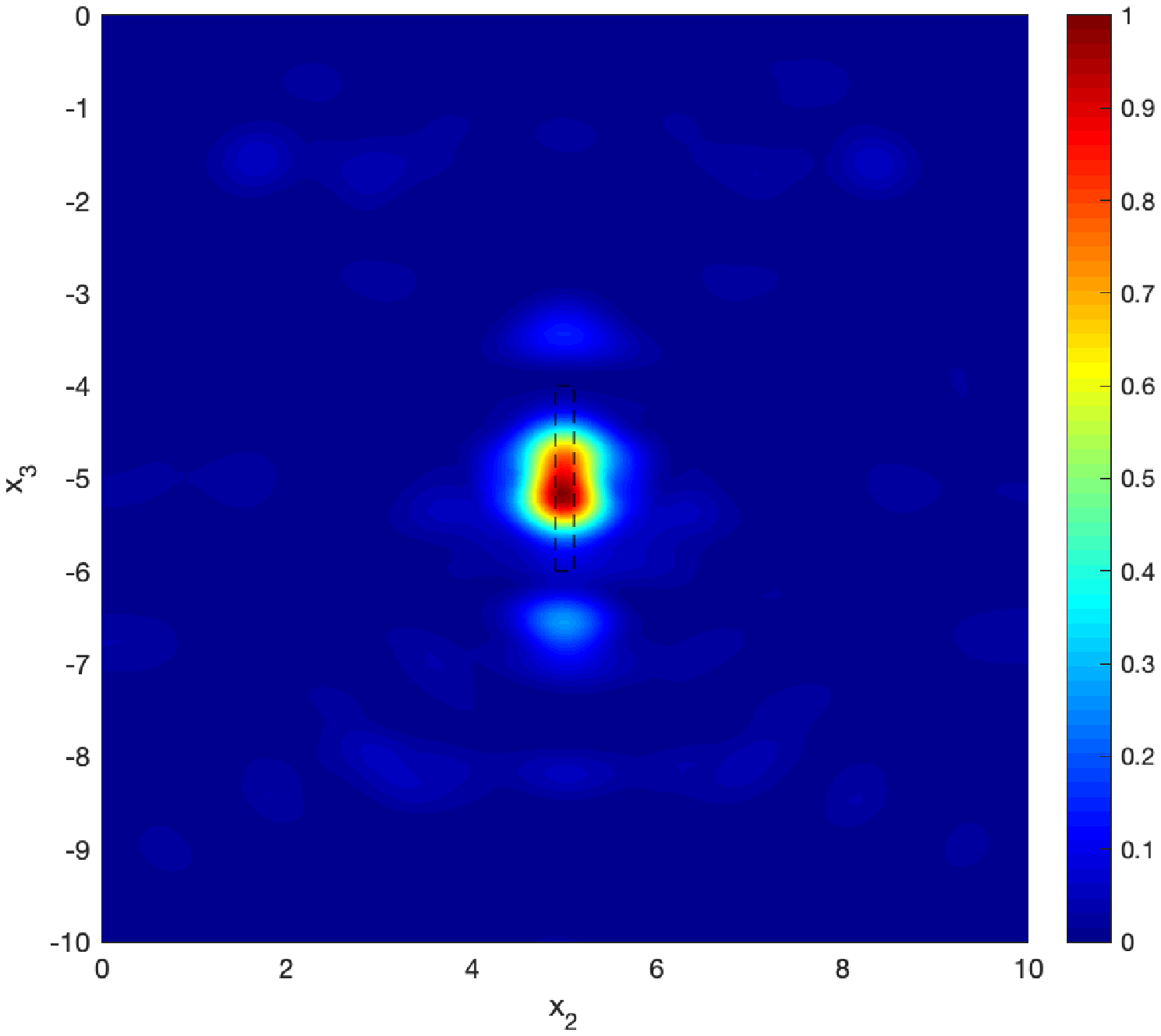}\includegraphics[width=0.33\linewidth]{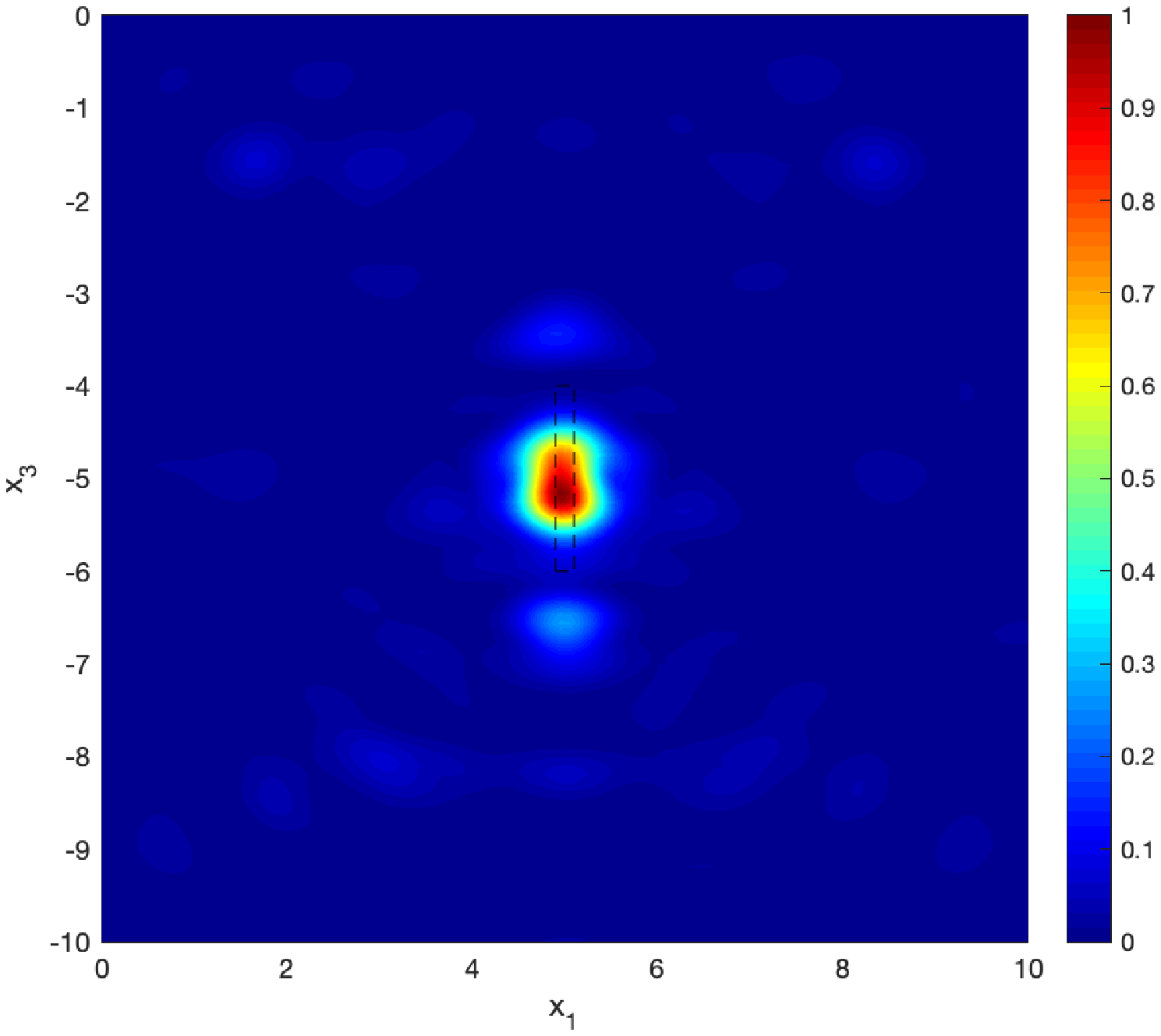}
\includegraphics[width=0.33\linewidth]{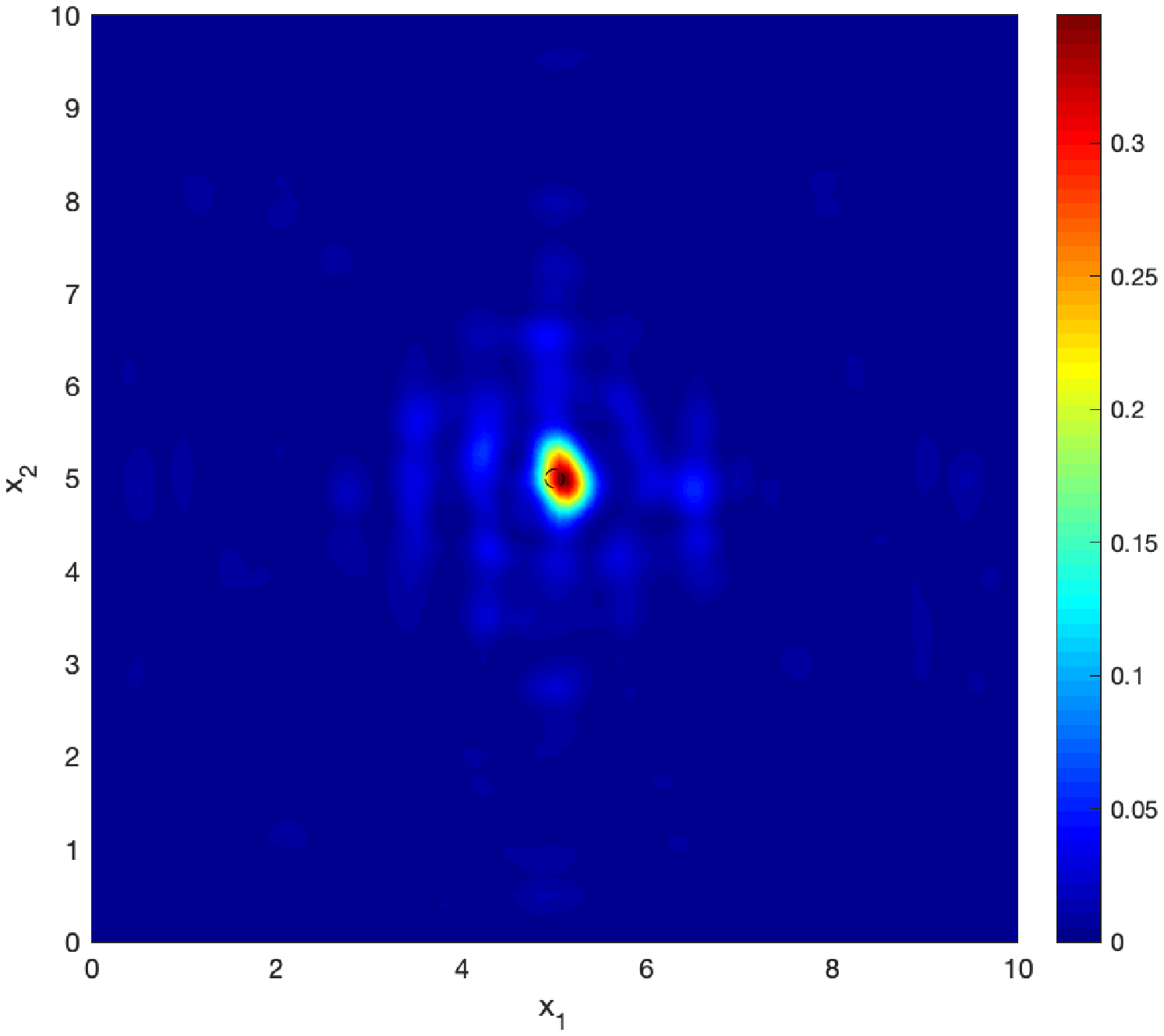}\includegraphics[width=0.33\linewidth]{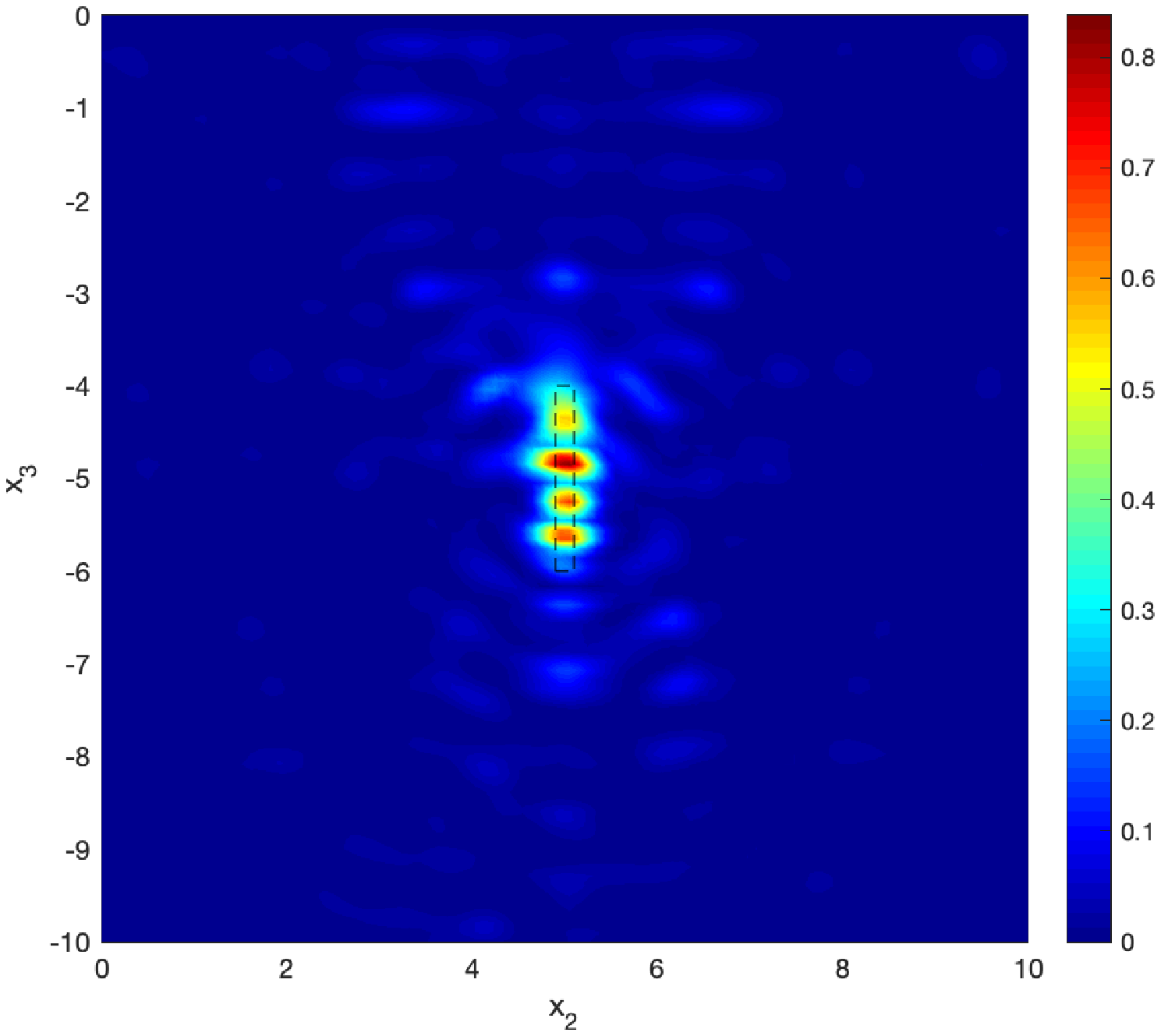}\includegraphics[width=0.33\linewidth]{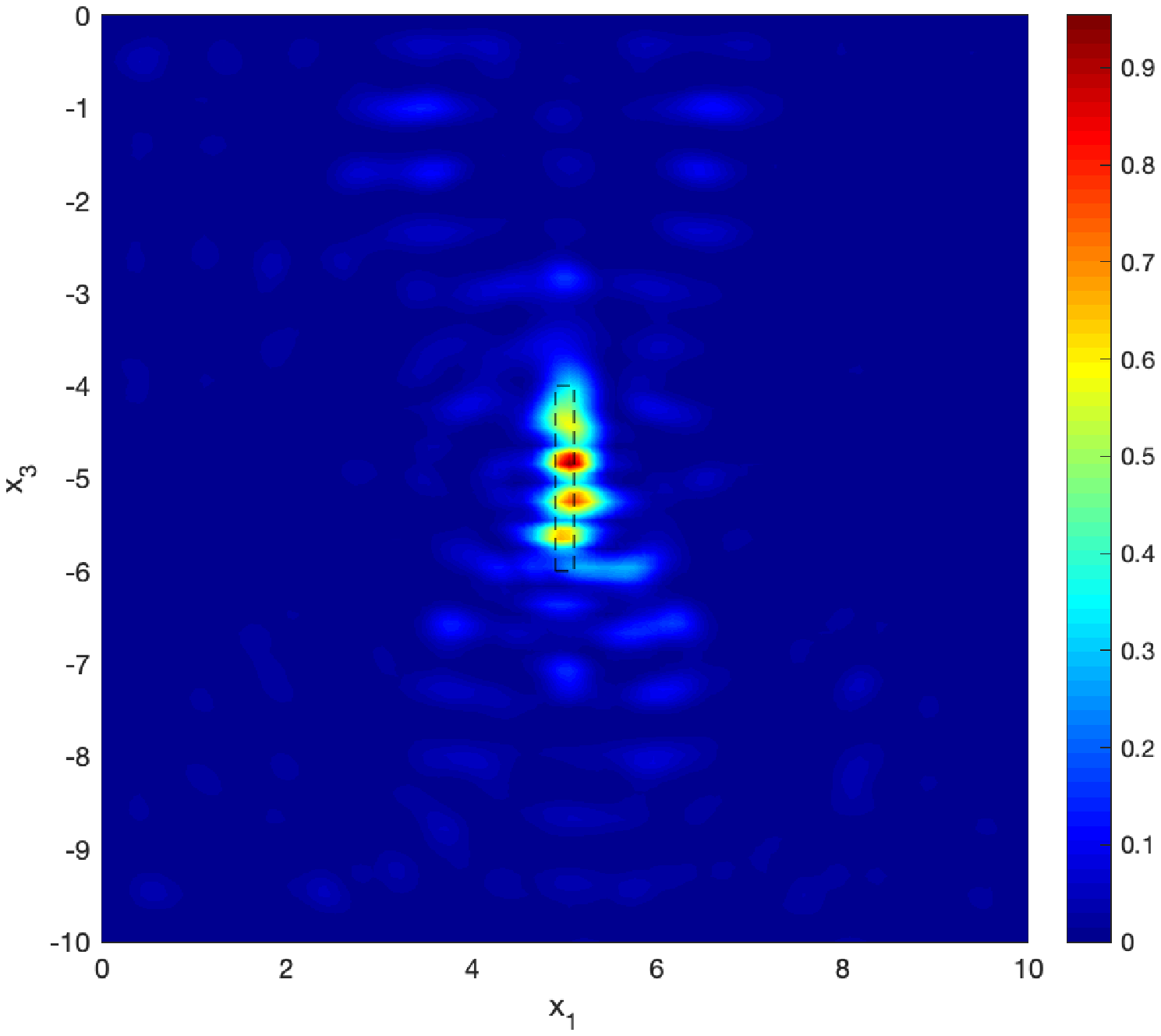}
     \caption{
     \linespread{1}
Image of a cylindrical scatterer. Top left: exact. Top middle: $k=3$; three dimensional image using iso-surface plot with iso-value $0.4$. Top right: $k=5$; three dimensional image using iso-surface plot with iso-value $0.3$.
In the middle ($k=3$) and bottom ($k=5$) row, we plot the cross section images where the exact geometry is indicated by the dashed line. Middle/bottom row left: $x_1x_2$-cross section image.  Middle/bottom row middle: $x_2x_3$-cross section image. Middle/bottom row right: $x_1x_3$-cross section image. 
     } \label{fig lcyl}
    \end{figure}

Another set of numerical examples is imaging of a scatterer consisting of two balls (top in Fig.~\ref{fig 3ball}) and a scatterer consisting of three balls (bottom in Fig.~\ref{fig 3ball}) respectively. The scatterer consisting of two balls is given by
\begin{equation*}
\{\bx: |\bx-(3,3,-5)|<0.5\} \cup \{\bx: |\bx-(7,7,-5)|<0.5\},
\end{equation*}
and the scatterer consisting of three balls is given by
\begin{equation*}
\{\bx: |\bx-(3,3,-5)|<0.5\} \cup \{\bx: |\bx-(7,7,-5)|<0.5\} \cup \{\bx: |\bx-(5,5,-5)|<0.5\}.
\end{equation*}
We can clearly distinguish three balls from two balls.
    \begin{figure}[ht!]
        \includegraphics[width=0.33\linewidth]{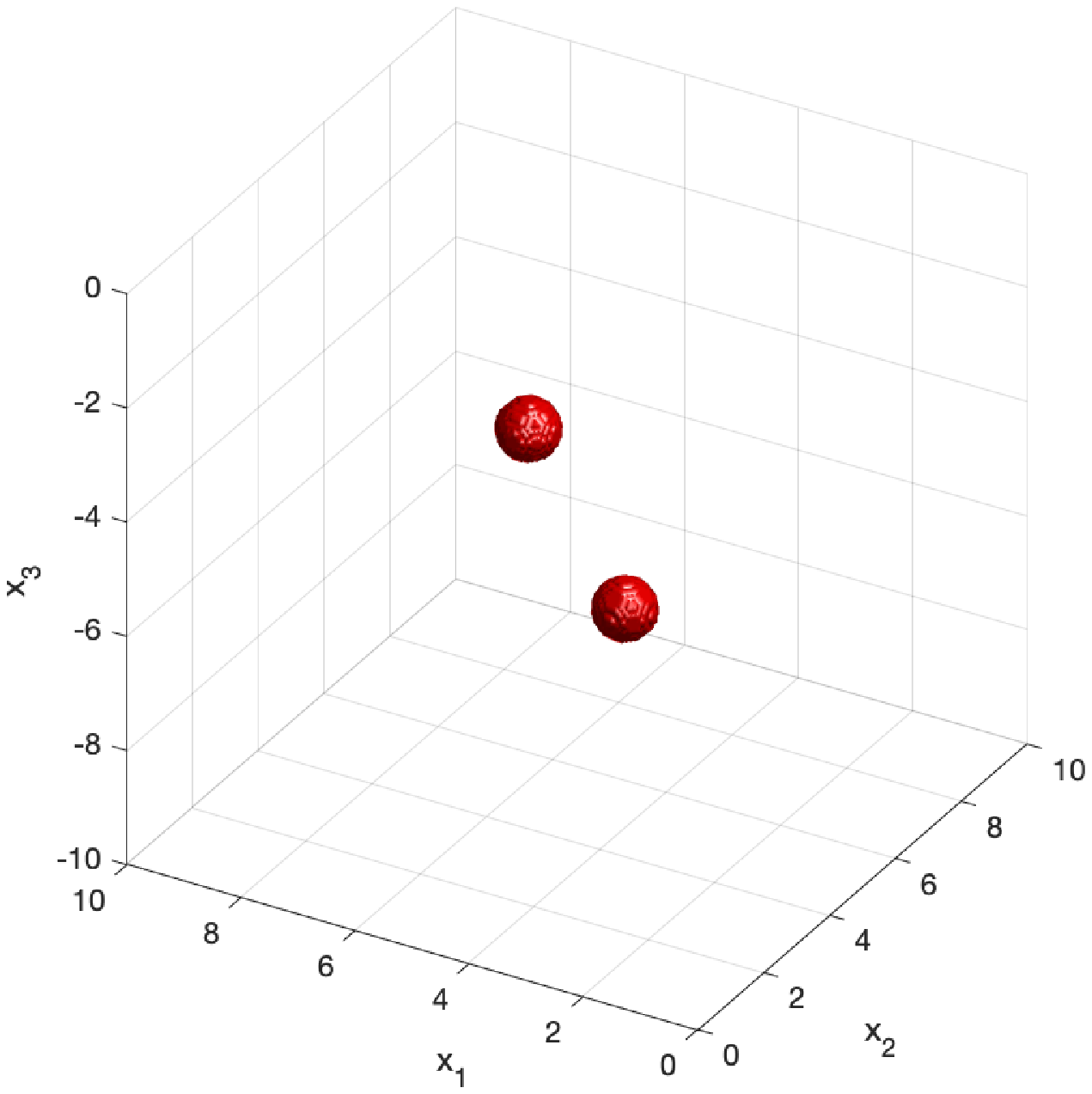}\includegraphics[width=0.33\linewidth]{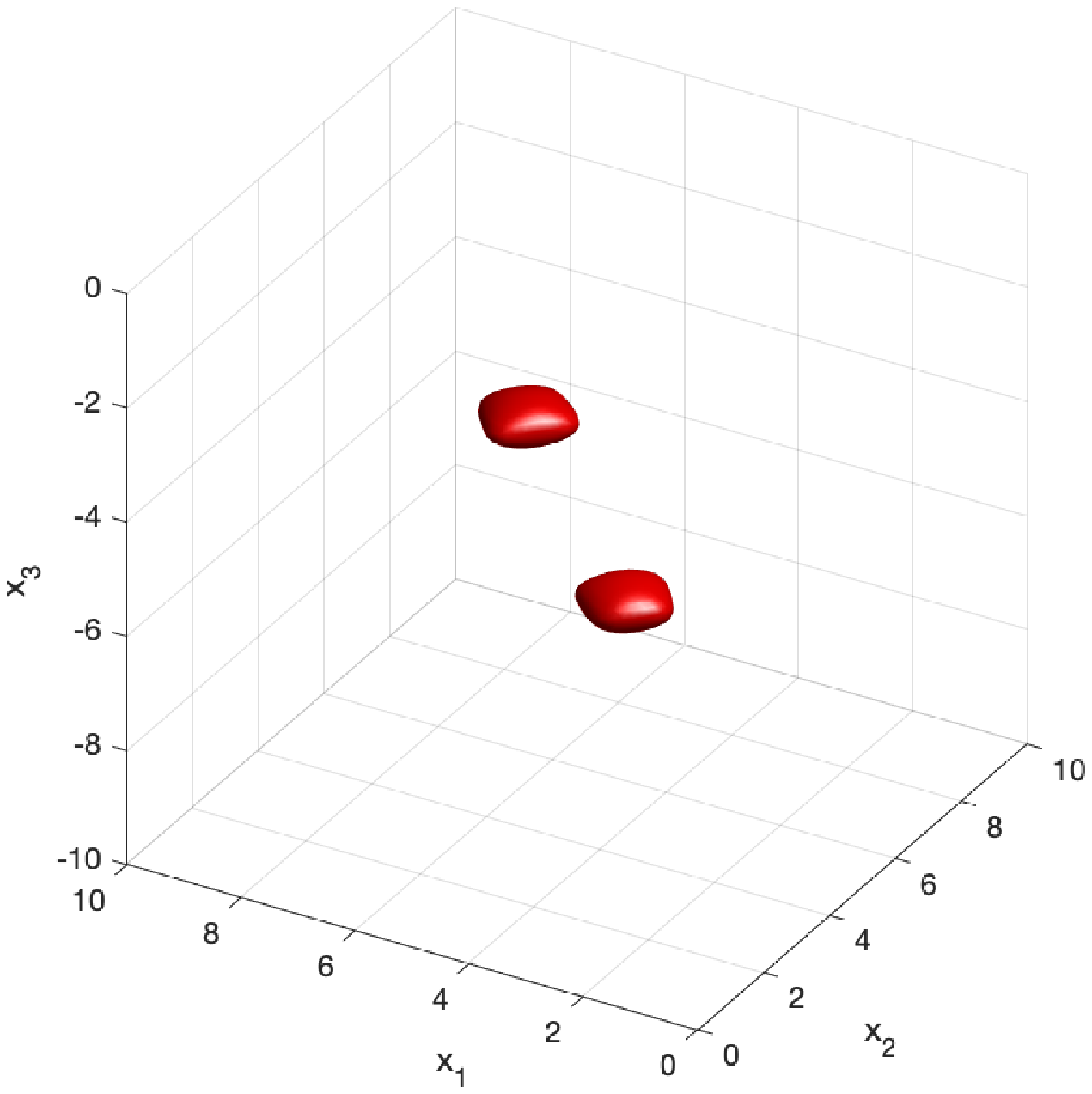}\includegraphics[width=0.33\linewidth]{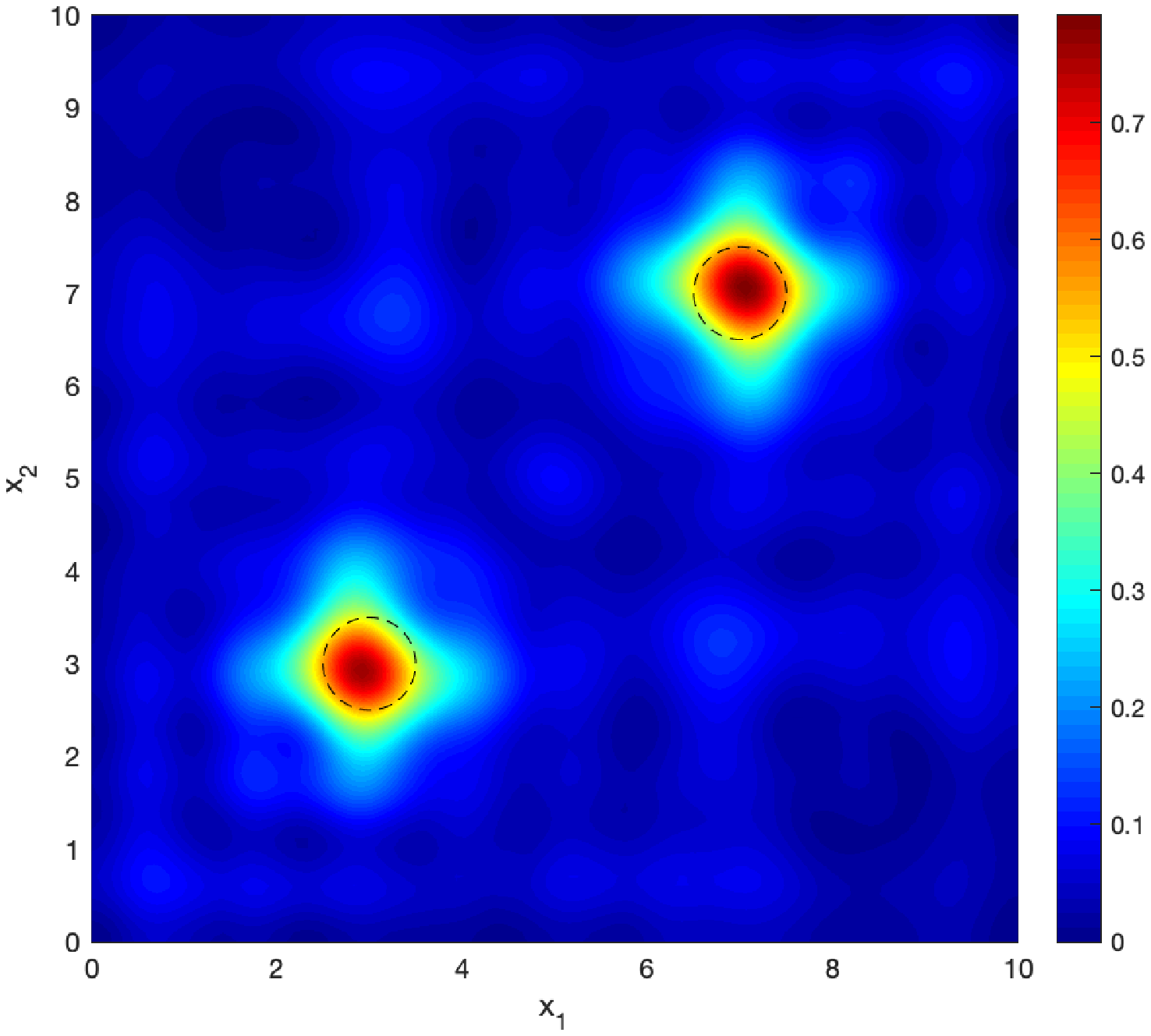}
    \includegraphics[width=0.33\linewidth]{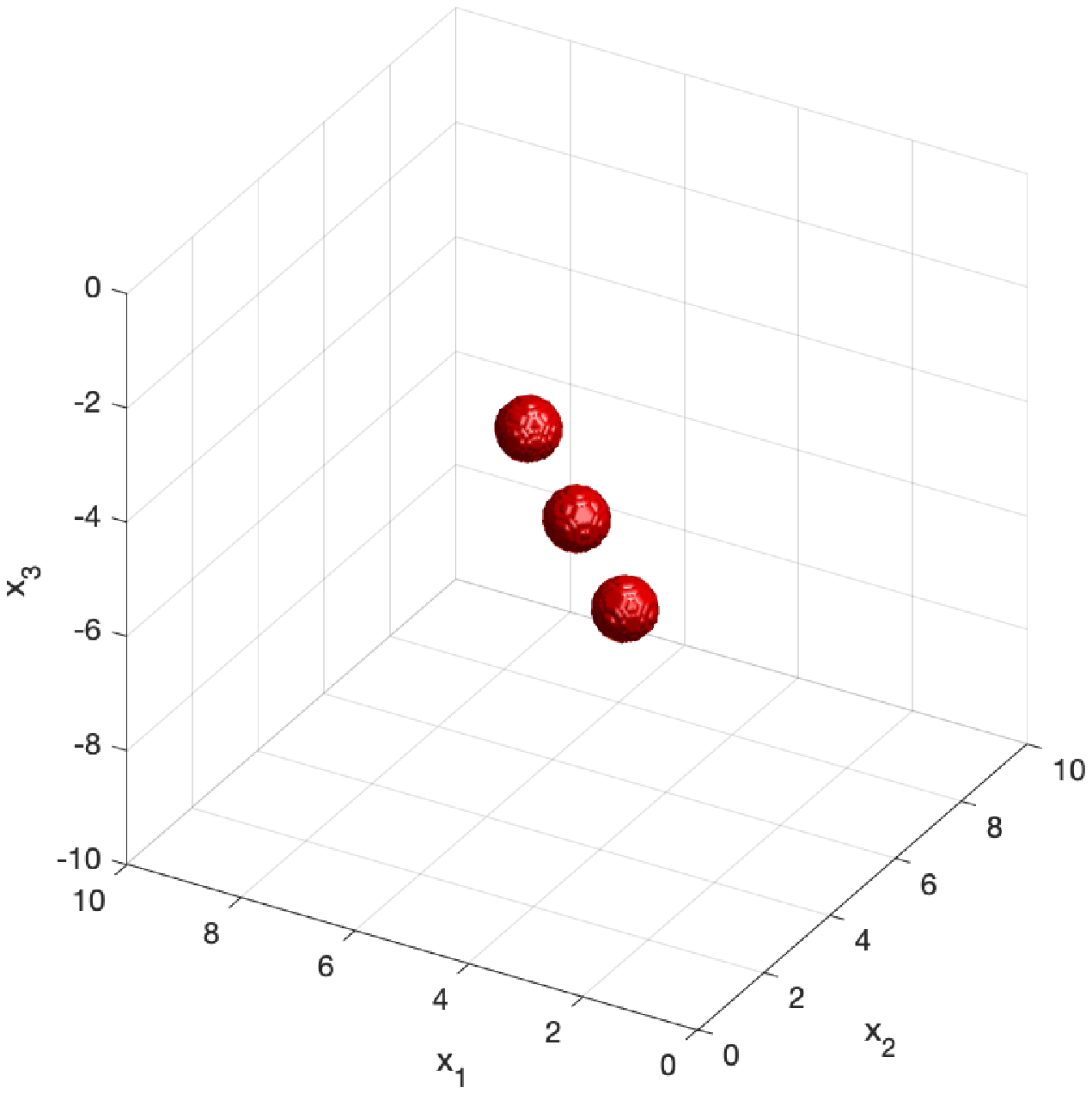}\includegraphics[width=0.33\linewidth]{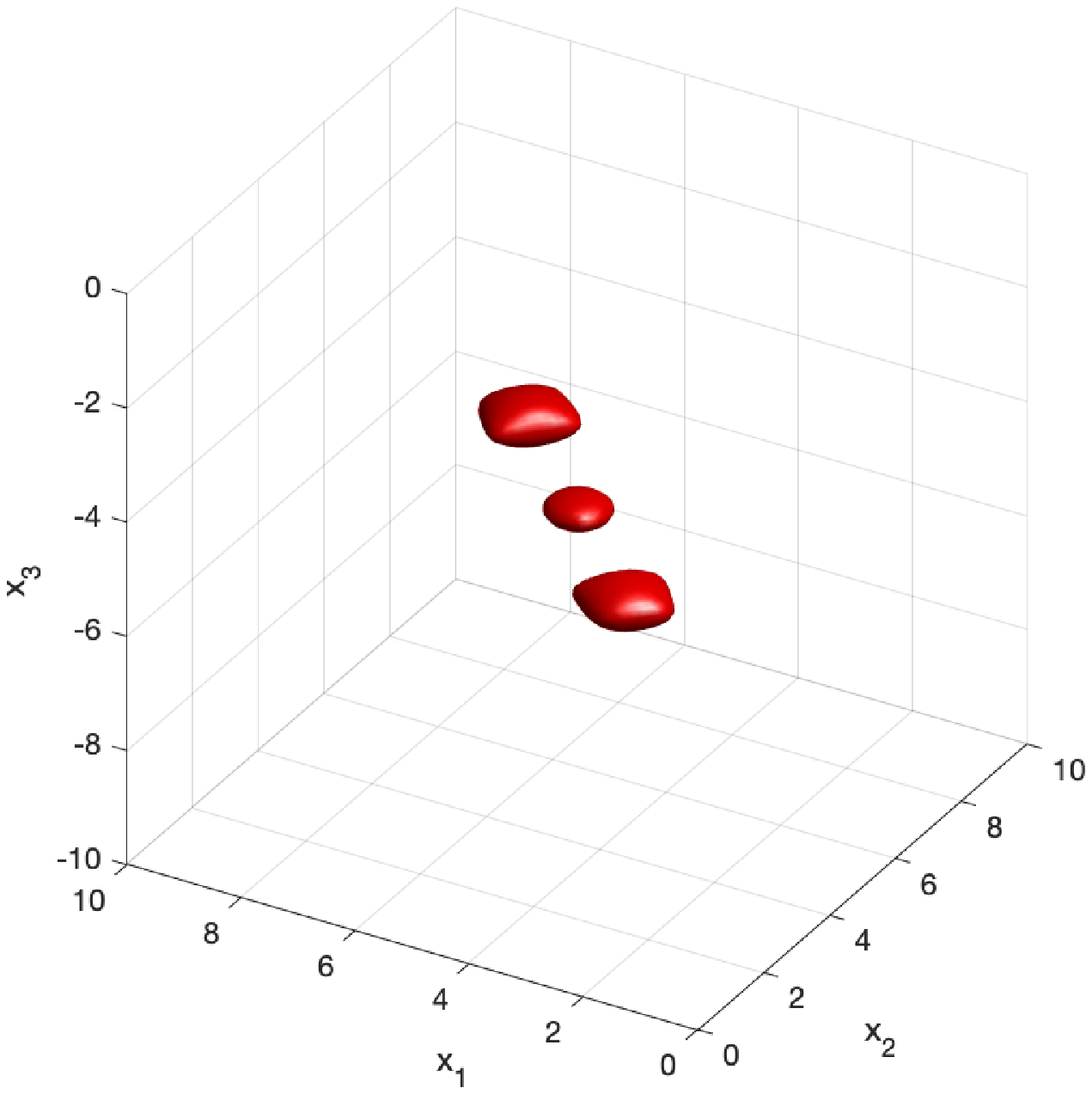}\includegraphics[width=0.33\linewidth]{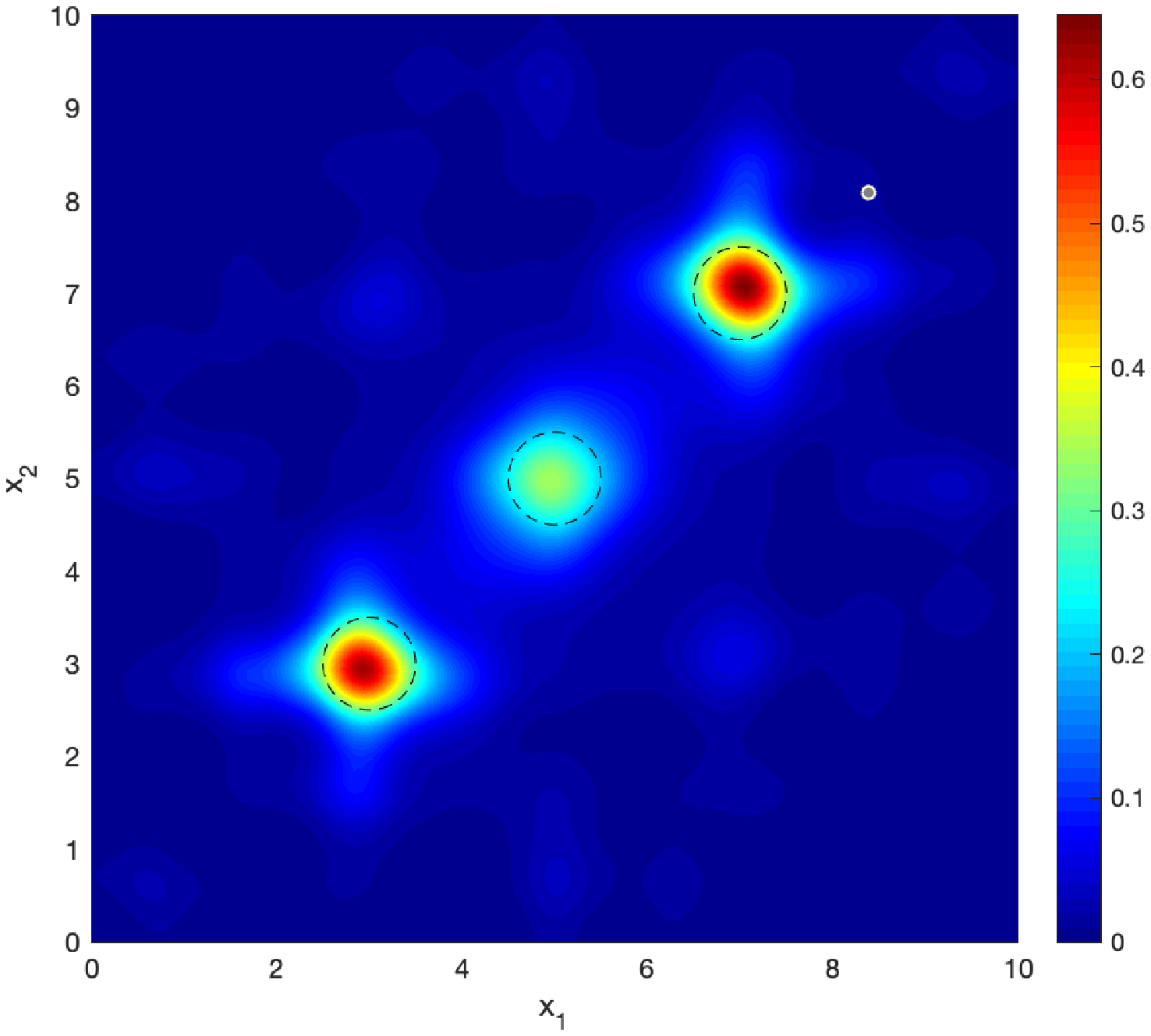}
     \caption{
     \linespread{1}
Top: image of two balls. Bottom: image of three balls. Left: exact. Middle: three dimensional image using iso-surface plot (with iso-value 0.2). Right: $x_1x_2$-cross section image.
     } \label{fig 3ball}
    \end{figure}
    
{
The next set of numerical example is imaging of a larger scatterer using different wavenumbers. The scatterer is a  ball given by
\begin{equation*}
\{\bx: |\bx-(5,5,-5)|<2\}.
\end{equation*}
We note that a similar ball (indeed they are of the same size when the waveguide length is scaled to the same) was also reported in \cite{monk2019near}. In Fig. \ref{fig lball}, the left plot is the exact ball. The middle plot is imaging with $k=1$ (in which case $M=12$ and $N=6$, and we use iso-surface plot with iso-value $0.6$). The right plot is imaging with $k=3$ (in which case $M=82$ and $N=64$, and we use iso-surface plot with iso-value $0.5$). We observe that one can quickly locate the ball using $k=1$ and obtain better images using higher wavenumber (i.e. more propagating modes).
}
    \begin{figure}[ht!]
    \includegraphics[width=0.33\linewidth]{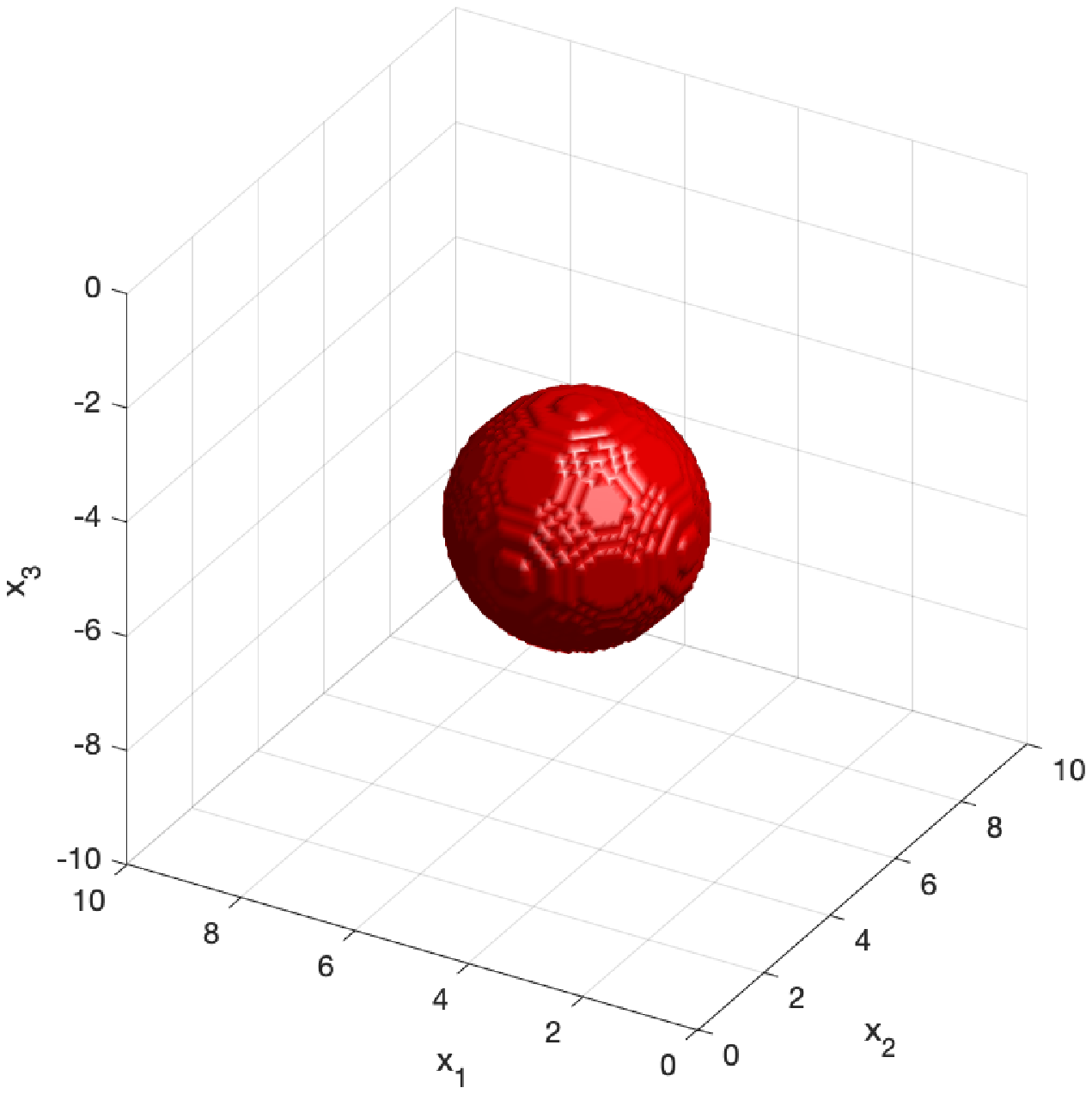}\includegraphics[width=0.33\linewidth]{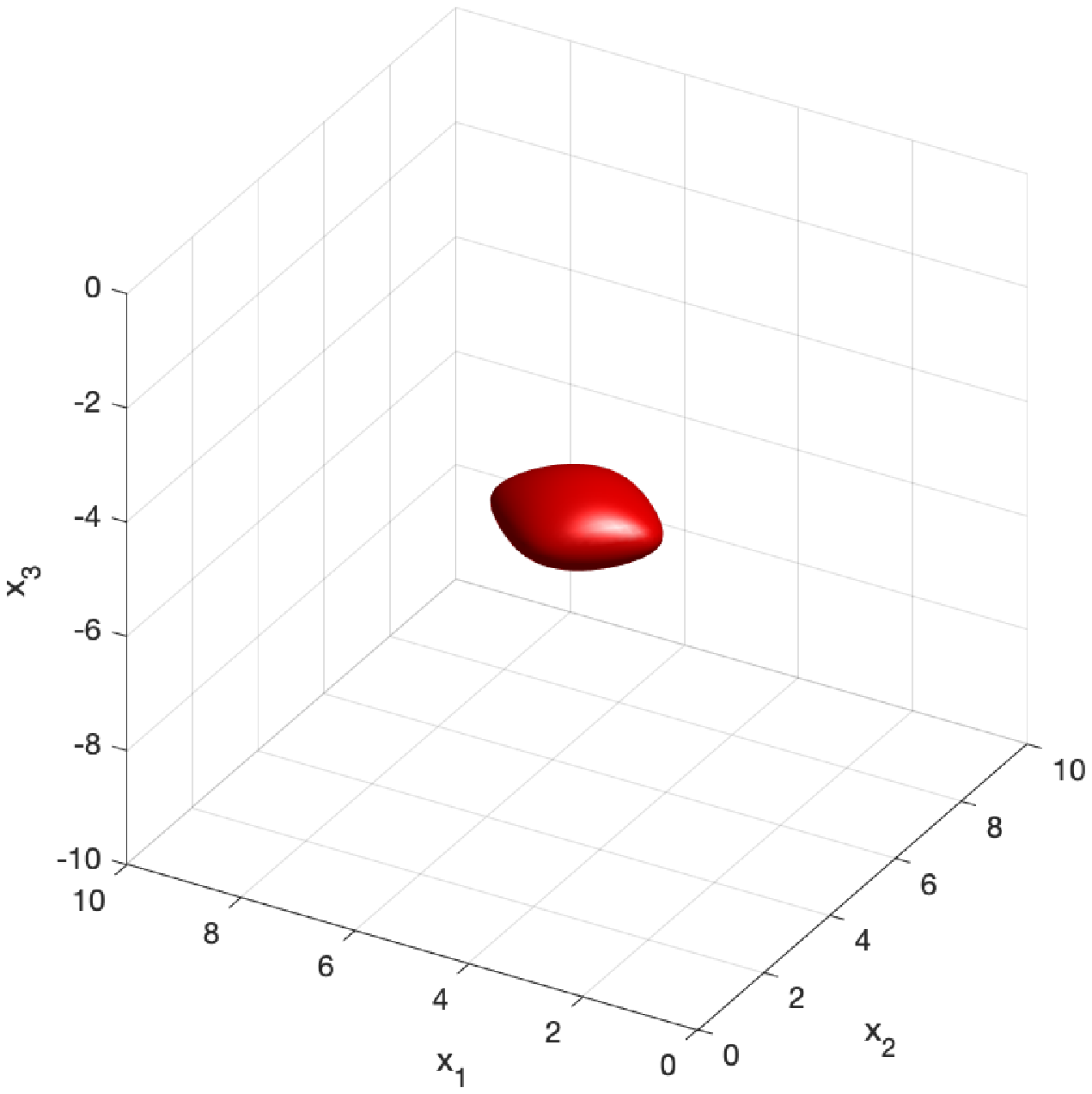}\includegraphics[width=0.33\linewidth]{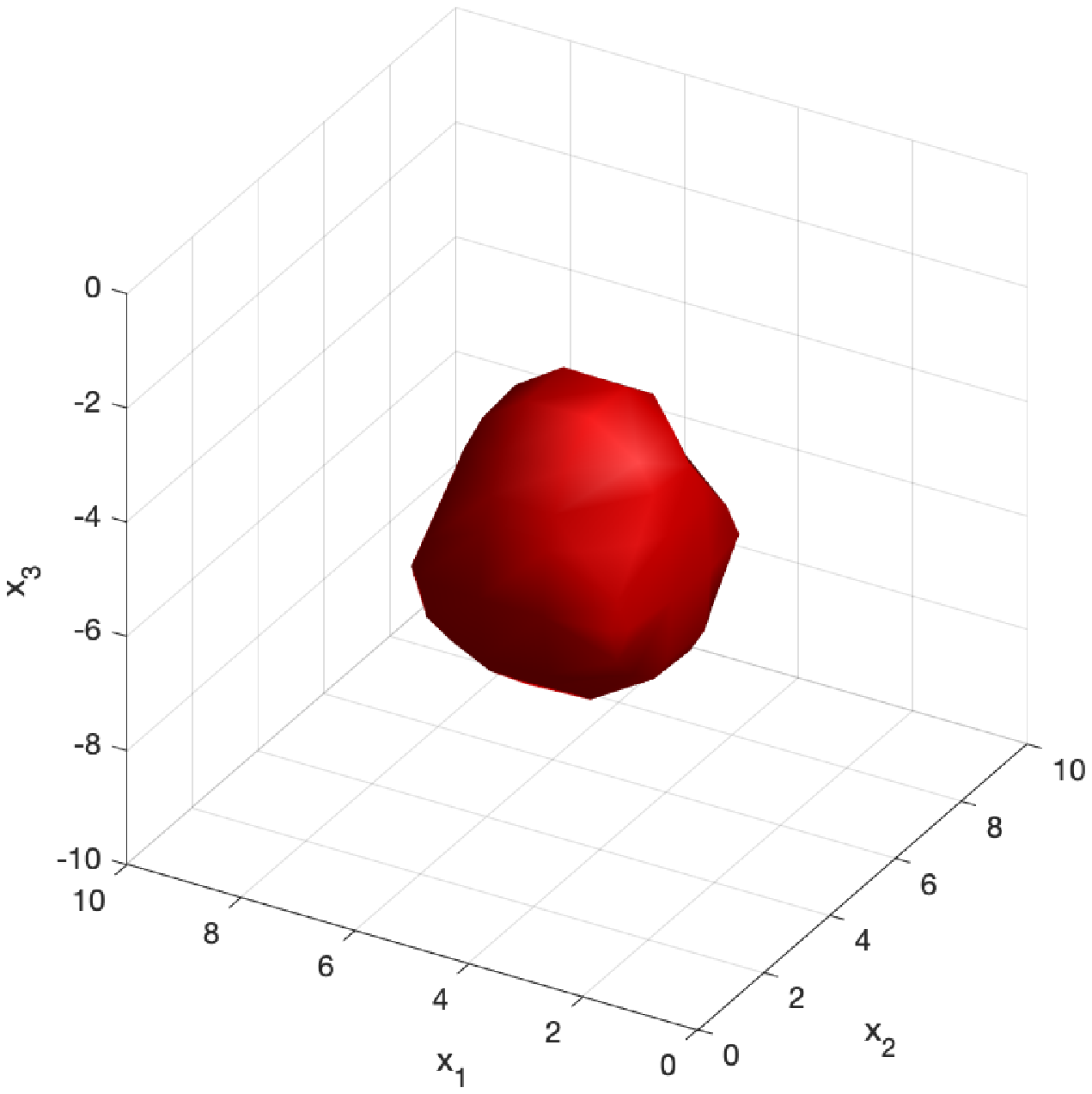}
     \caption{
     \linespread{1}
Image of a large ball. Left: exact. Middle: three dimensional image when $k=1$. Right:  three dimensional image when $k=3$.
     } \label{fig lball}
    \end{figure}

{
Finally we illustrate the robustness of the imaging method with respect to noises in Fig. \ref{fig noise}. We plot the iso-surface (with the same iso-value 0.4) image of the L-shape scatterer with $10\%$ and $30\%$ noise respectively. Our imaging method is observed to be robust with respect to noises.
}
    
        \begin{figure}[ht!]
        \includegraphics[width=0.33\linewidth]{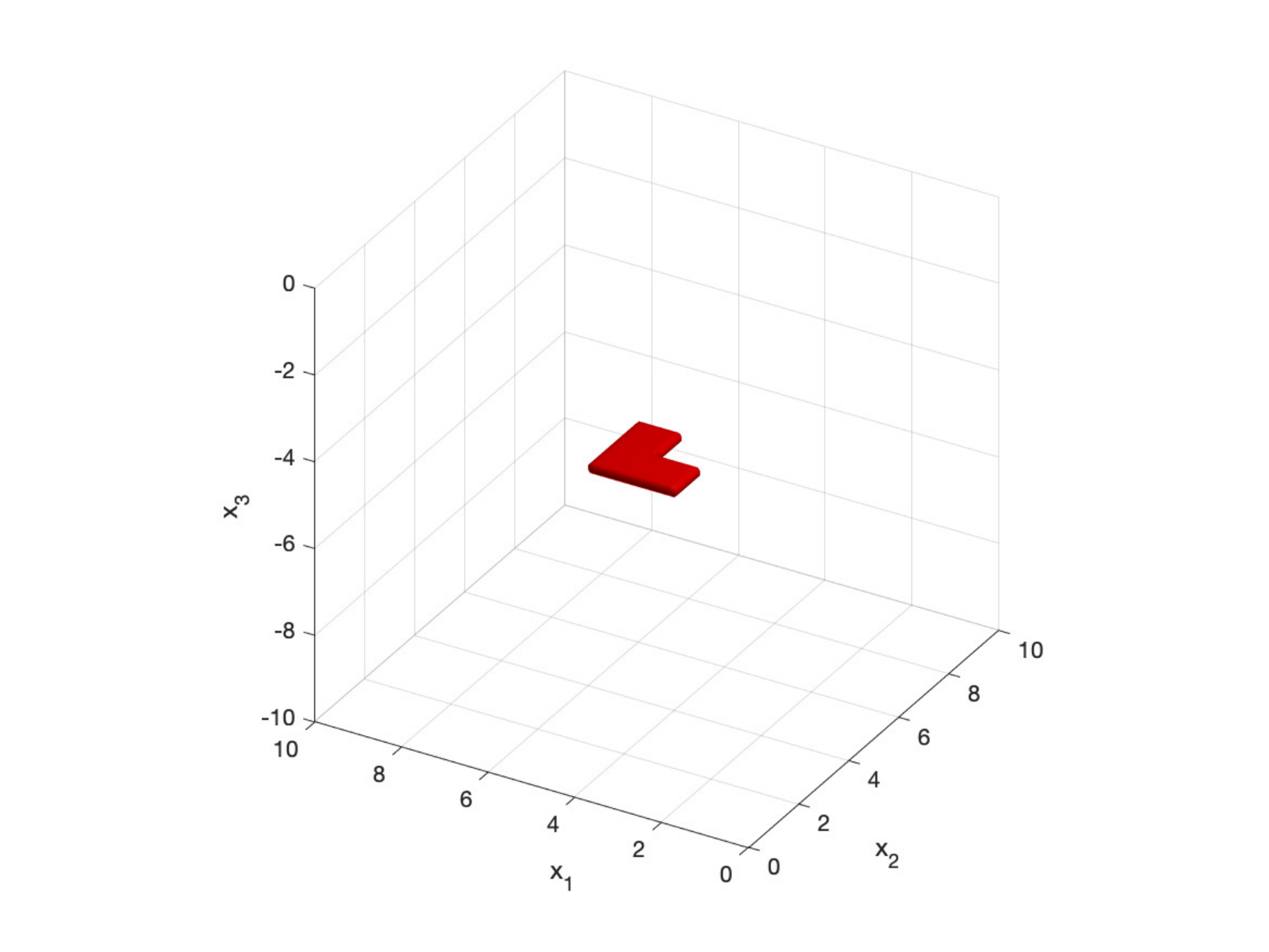}
        \includegraphics[width=0.33\linewidth]{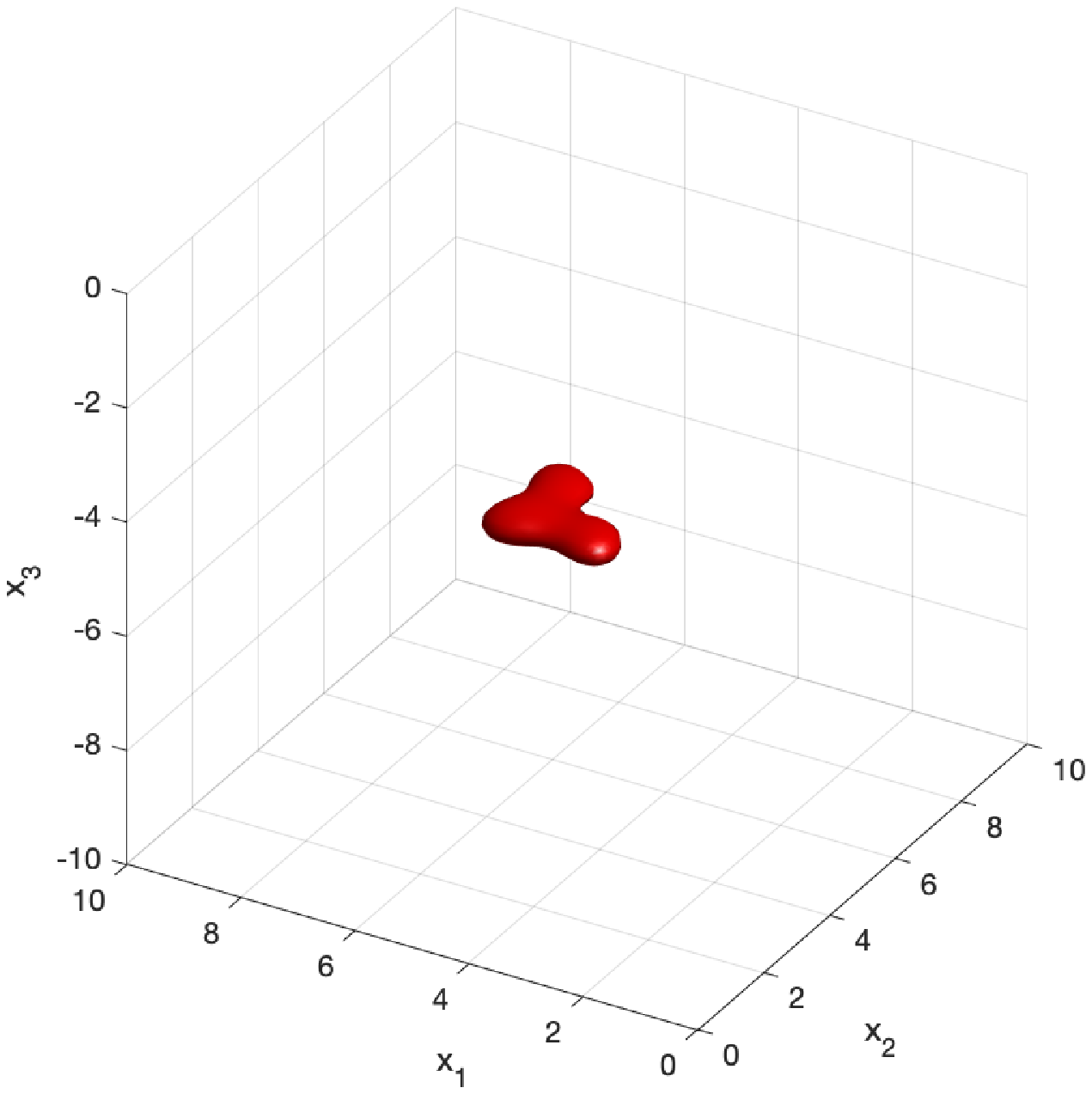}\includegraphics[width=0.33\linewidth]{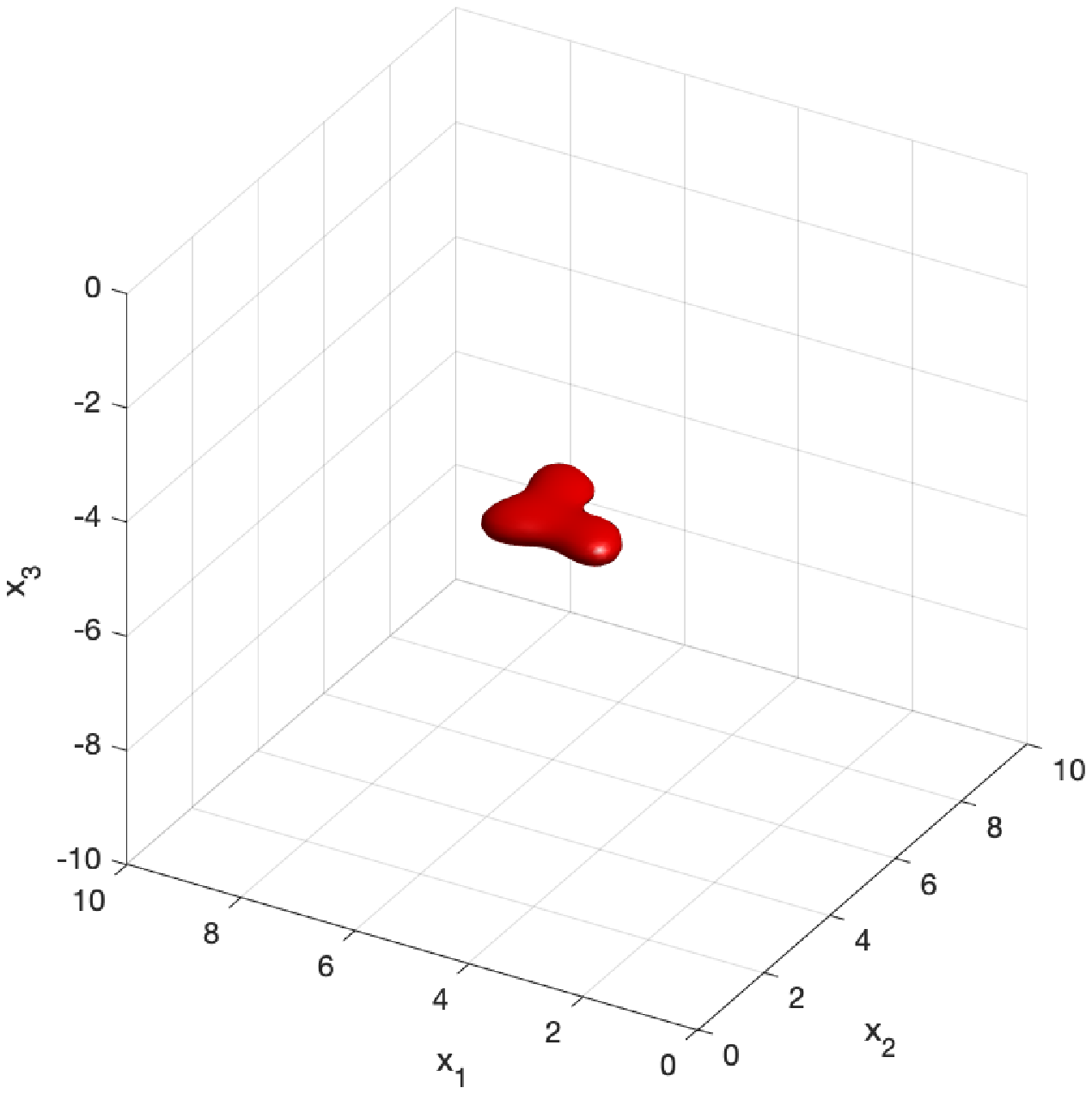}
     \caption{
     \linespread{1}
Iso-surface (with the same iso-value 0.4) image of the L-shape scatterer. Left: exact. Middle: $10\%$ noise. Right: $30\%$ noise.
     } \label{fig noise}
    \end{figure}

\section{Conclusions}\label{conclusion}
The analysis and numerical examples show that our sampling type method is capable to image extended scatterers in the electromagnetic waveguide.  Based on  integrating the measurements and a known function over the measurement surface directly, our sampling method is robust, computationally efficient, and does not require a priori estimate on the measurement noise. Our future work includes sampling type methods for waveguide in the time domain, and we expect that it will improve the image by using multi-frequency data directly in the time domain.


\medskip
Received xxxx 20xx; revised xxxx 20xx.
\medskip

\end{document}